\documentclass[a4paper,12pt]{amsart}

\usepackage{amsmath,amssymb,amsfonts,mathrsfs,amsthm,stmaryrd,mathtools,enumerate}
\usepackage[all]{xy}
\usepackage{url}
\usepackage[colorlinks=true]{hyperref}
\usepackage[nobysame,alphabetic,initials,msc-links]{amsrefs}
\usepackage{cleveref}

\DefineSimpleKey{bib}{arxiv}
\newcommand\myurl[1]{\url{#1}}

\def\MR#1{\quad \href{http://www.ams.org/mathscinet-getitem?mr=#1}{MR#1}}

\DefineSimpleKey{bib}{how}
\renewcommand{\eprint}[1]{#1}
\BibSpec{misc}{%
  +{}{\PrintAuthors}  {author}
  +{,}{ \textit}      {title}
  +{,}{ }             {how}
  +{}{ \parenthesize} {date}
  +{,} { available at \eprint}        {eprint}
  +{,}{ available at \url}{url}
  +{,}{ }             {note}
  +{.}{}              {transition}
}

\newtheorem{thrm}{Theorem}[section]
\newtheorem{prop}[thrm]{Proposition}
\newtheorem{coro}[thrm]{Corollary}
\newtheorem{lemm}[thrm]{Lemma}

\theoremstyle{definition}
\newtheorem{defn}[thrm]{Definition}
\newtheorem{exam}[thrm]{Example}

\newtheorem{rema}[thrm]{Remark}

\newcommand{\introthm}[3]{

\vspace{0.8em}
\noindent\textbf{#1 #2.}\hspace{-1ex}
\textit{#3}
\vspace{0.8em}

}

\renewcommand{\bar}[1]{\overline{#1}}
\renewcommand{\hat}[1]{\widehat{#1}}

\newcommand{\Z}{\mathbb{Z}}

\newcommand{\C}{\mathbb{C}}

\newcommand{\G}{\mathbb{G}}
\newcommand{\K}{\mathbb{K}}
\renewcommand{\H}{\mathbb{H}}

\newcommand{\eps}{\varepsilon}
\newcommand{\lbd}{\lambda}
\renewcommand{\phi}{\varphi}
\newcommand{\map}[3]{#1\colon#2 \longrightarrow #3}

\newcommand{\abs}[1]{\lvert #1 \rvert}
\newcommand{\nor}[1]{\lVert #1 \rVert}
\newcommand{\ip}[1]{\left\langle #1 \right\rangle}

\newcommand{\cl}[1]{\mathcal{#1}}

\newcommand{\tensor}{\otimes}
\newcommand{\fin}{\mathrm{f}}
\newcommand{\astar}{$\ast$}
\newcommand{\cstar}{C*}
\newcommand{\hyp}{\mathrm{hyp}}
\newcommand{\HypGrp}{\mathbf{HypGrp}}
\newcommand{\tend}{\textendash}
\newcommand{\gbag}{\Gamma'\backslash\Gamma}

\newcommand{\Kac}{\mathrm{Kac}}
\newcommand{\uni}{\mathrm{uni}}

\newcommand{\bigast}{\mathop{\scalebox{1.5}{\raisebox{-0.2ex}{$\ast$}}}}

\let\mod\relax
\DeclareMathOperator{\stker}{\mathrm{StKer}}
\DeclareMathOperator{\FPdim}{\mathrm{dim_{FP}^{nor}}}

\DeclareMathOperator{\id}{\mathrm{id}}

\DeclareMathOperator{\Index}{\mathrm{Index}}

\DeclareMathOperator{\hilb}{\mathrm{Hilb}}
\DeclareMathOperator{\rep}{\mathrm{Rep}}
\DeclareMathOperator{\mod}{\mathrm{Mod}}
\DeclareMathOperator{\corr}{\mathrm{Corr}}
\DeclareMathOperator{\rfin}{\mathrm{rf}}
\DeclareMathOperator{\irr}{\mathrm{Irr}}
\DeclareMathOperator{\onb}{\mathrm{ONB}}

\DeclareMathOperator{\ev}{\mathrm{ev}}

\title[Finite index quantum subgroups of DQGs]{Finite index quantum subgroups of discrete quantum groups}
\author{Mao Hoshino}
\address{Department of Mathematical Sciences, The University of Tokyo\\
Komaba 3-8-1, Tokyo \mbox{153-8914}, Japan}
\email{mhoshino@ms.u-tokyo.ac.jp}
\thanks{This work was supported by JSPS KAKENHI Grant Number JP23KJ0695 and WINGS-FoPM Program at the University of Tokyo.}

\begin{document}

\begin{abstract}
We show that finite index quantum subgroups of a discrete quantum group
are induced from finite index quantum subgroups of the unimodularization. As an application, we classify
all finite index quantum subgroups of free products of the duals of
connected simply-connected compact Lie groups.
We also put proofs for some fundamental facts on finite index
right coideals of compact quantum groups.
\end{abstract}

\maketitle

\section{Introduction}

Starting with S.L. Woronowicz's discovery of the notion of
compact quantum group (\cite{MR901157,MR1616348}),
the operator algebraic framework of quantum group
has been developed by a number of researchers. Many
classical notions for locally compact groups are brought
into the theory of locally compact quantum group in successful ways,
including the definition of locally compact quantum group itself
given by J. Kustermans and S. Vaes in \cite{MR1832993}.

For the notion of closed quantum subgroup of locally compact
quantum group, there are two definitions proposed by S. Vaes (\cite[Definition 2.5]{MR2182592})
and S.L. Woronowicz (\cite[Definition 3.2]{MR2980506}).
It remains open whether these definitions are equivalent in general.
However, it was shown in \cite{MR2980506} that Vaes' definition is stronger than or equivalent to Woronowicz's definition, and they coincide in
certain cases, including the setting of discrete quantum groups.
This allows us to have the unified definition of
quantum subgroups of discrete quantum groups,
which can also be captured by the representation theory
of the dual compact quantum groups. Indeed it might be well-known
to experts that there is a one-to-one correspondence
between quantum subgroups of a discrete quantum group $\Gamma$ and
the set of non-empty subsets of $\irr \hat{\Gamma}$
which are closed under tensor products and conjugate.
See \cite[Lemma 4.1]{MR4742819} for example.

In some specific cases, quantum subgroups of a given discrete
quantum group are explicitly classified. In \cite{freslon2024discretequantumsubgroupsfree} A. Freslon and M. Weber
classified all quantum subgroups of the
duals of the free unitary quantum groups and give explicit presentations
of them. The classification result for the free products
of such discrete quantum groups is also given by K. Kitamura
and the author (\cite{MR4824928}). This classification describes quantum subgroups in terms of Cayley graphs of free groups.

The first objective of this paper is to provide some fundamental facts on
finite index right coideals of compact quantum groups, which can be applied
to quantum subgroups of discrete quantum groups.
Expecially we show that some natural finiteness conditions on
right coideals are equivalent:

\introthm{Proposition}{\ref{prop:characterizations of finite index right coideals}}{
Let $\G$ be a compact quantum group and
$B$ be a right coideal of $C(\G)$. The algebraic cores of $C(\G)$
and $B$ are denoted by $\cl{O}(\G)$ and $\cl{B}$.
The following conditions are equivalent:
\begin{enumerate}
 \item A right $\cl{B}$-module $\cl{O}(\G)$ is finitely generated.
 \item There is a conditional expectation $\map{\cl{E}}{\cl{O}(\G)}{\cl{B}}$ which is of finite index and preserves the right coactions of $\cl{O}(\G)$.
 \item There is a $\G$-equivariant conditional expectation
$\map{E}{C(\G)}{B}$ which is of finite index.
 \item $\dim_{\C} B^{\perp} < \infty$, where $B^{\perp} \subset \ell^{\infty}(\hat{\G})$ is the coideal orthogonal to $B$.
 \item The category $\G\text{-}\mod^{\fin}_B$ of finitely generated
$\G$-equivariant Hilbert $B$-modules
has only finitely many irreducible objects.
\end{enumerate}
In this case we also have $\Index^s E = \dim_{\C}B^{\perp} = \FPdim \G\text{-}\mod_B$, which is denoted by $[C(\G):B]$.
}

We say that a right coideal $B$ is \emph{of finite index} if it satisfies
these conditions.
I would like to emphasize that the conditions (i), (iv) and (v) look weaker than (ii) and (iii) since we do not assume the existence of
a$\G$-equvariant conditional expectation onto $B$ in these conditions.
Actually we need to use some non-trivial results on quantum groups
to prove the equivalence, namely Vaes's implementation
theorem for actions (\cite{MR1814995})
and a special property of the maximal Kac quantum subgoup of
a compact quantum group (\cite{MR2210362}).

We also obtain the following imprimitivity theorem on finite index right coideals. Analogous results are known for different kinds of actions of the Drinfeld-Jimbo deformations. (\cite[Corollary 3.8, Corollary 4.7, Theorem 4.8]{MR4728596})

\introthm{Proposition}{\ref{prop:imprimitivity}}
{
Let $\G$ be a compact quantum group and $\map{q}{C(\G)}{C(\G_\Kac)}$
be the maximal Kac quantum subgroup.
Then, for any finite index right coideal $B$ of $C(\G)$,
there is a finite index right coideal $B_{\Kac}$ of $C(\G_{\Kac})$
such that $\cl{B} = q^{-1}(\cl{B}_{\Kac})$,
where $\cl{B}$ and $\cl{B}_{\Kac}$ are the algebraic cores of $B$
and $B_{\Kac}$ respectively. Moreover we have $[C(\G):B] = [C(\G_{\mathrm{Kac}}):B_{\mathrm{Kac}}]$.
}

The second objective of this paper is
to reveal a significant restriction on finite index quantum subgroups
of discrete quantum groups. We say that a discrete quantum
subgroup $\Gamma'$ of $\Gamma$ is \emph{of finite index}
if $C(\hat{\Gamma'})$ is a finite index right coideal of $C(\hat{\Gamma})$.
Then Proposition \ref{prop:imprimitivity} has the following consequence. 
\introthm{Theorem}{\ref{thrm:reduction to the unimodularization}}
{
Let $\Gamma$ be a discrete quantum group and $\map{q}{\Gamma}{\Gamma_{\uni}}$ be its unimodularization. Then there is a natural 1-to-1
correspondence between the following:
\begin{itemize}
 \item finite index quantum subgroups of $\Gamma$,
 \item finite index subhypergroups of $\stker q_{\hyp}\backslash \Gamma_{\hyp}/\stker q_{\hyp}$.
\end{itemize}
}
Here we omit the definition of $\stker q_{\hyp}\backslash \Gamma_{\hyp}/\stker q_{\hyp}$, which is a hypergroup in general.
In some interesting cases it is a group which gives a grading on
$\rep^{\fin}\hat{\Gamma}$.

This result provides a general strategy for classifing
finite index discrete quantum subgroups of a discrete quantum group $\Gamma$: find a discrete quantum group $\Gamma'$ whose associated
hypergroup $\Gamma'_{\hyp}$ is isomorphic to $\Gamma_{\hyp}$
while ensuring its unimodularization $\Gamma'_{\uni}$ is small enough.
Actually we have the following reduction theorem on
finite index quantum subgoups of the free products of the dual of
connected simply-connected compact Lie groups
as a successful example of this strategy.
\introthm{Corollary}{\ref{coro:findexsqg}}
{
Let $(K_{\lbd})_{\lbd \in \Lambda}$ be a family
of connected simply-connected compact Lie groups.
Then there is a bijection between the set of finite index quantum subgroups of $\bigast_{\lbd \in \Lambda} \hat{K_{\lbd}}$ and the set of finite index subgroups
of $\bigast_{\lbd \in \Lambda} P_{\lbd}/Q_{\lbd}$, where $P_{\lbd}$
and $Q_{\lbd}$ are the weight lattice and the root lattice
associated to $K_{\lbd}$.
}

This paper is organized as follows. In Section 3, we investigate
right coideals of compact quantum groups with some finiteness conditions. We first show Proposition \ref{prop:characterizations of finite index right coideals}, with
part of its proof postponed to Section 4, where the notion of
maximal Kac quantum subgroup is explained. The remaining part of
Section 3 is devoted to some formulae on index of quantum subgroup of discrete quantum group.

In Section 4 we prove some imprimitivity theorems on finite index right
coideals and finite index quantum subgroups, after giving a brief
review on the maximal Kac quantum subgroup.

In Section 5 we discuss an application of the imprimitivity theorem
on finite index quantum subgroups to free products
of discrete quantum groups.

\section{Preliminaries}
For a Hilbert space $H$, its inner product $\ip{\tend,\tend}$ is
$\C$-linear in the second argument. An element $\xi \in H$ is regarded as an operator from $\C$ to $H$.

For a $\C$-vector space $V$, its algebraic dual is denoted by $V^{\vee}$. The image of $v \in V$ under the canonical embedding $V \subset V^{\vee}$ is denoted by $\mathrm{ev}_v$

In this paper, the symbol $\textendash\tensor\textendash$ denotes
the spatial tensor product of C*-algebras, the algebraic tensor
product of $\C$-vector spaces and the exterior tensor product of 
Hilbert C*-modules.

\subsection{Compact quantum groups and their morphisms}\label{subsec:cqg}
In this subsection, we give a brief review on compact quantum groups.
See \cite{MR3204665} for detailed discussions.

A \emph{compact quantum group} is a pair $\G = (A,\Delta)$
of a non-zero unital C*-algebra $A$ and a \astar-homomorphism
$\map{\Delta}{A}{A\otimes A}$ which satisfies the following conditions:
\begin{itemize}
 \item the coassociativity $(\Delta\tensor\id)\Delta = (\id\tensor\Delta)\Delta$,
 \item the cancellation property $\bar{(A\otimes \C)\Delta(A)} = \bar{(\C\tensor A)\Delta(A)} = A\tensor A$.
\end{itemize}
In this case $A$ is denoted by $C(\G)$. A \emph{Haar state} on $\G$ is a
state $h$ on $C(\G)$ satisfying the bi-invariance property
$(h\tensor\id)\Delta(x) = (\id\tensor h)\Delta(x) = h(x)1_{\G}$.
Such a state always exists and is unique.
We say that $\G$ is \emph{reduced} if $h$ is faithful.
We only consider reduced compact quantum groups unless otherwise noted.

A \emph{unitary representation} of $\G$ is a pair
$\pi = (H_{\pi},U_{\pi})$ of a Hilbert space $H_{\pi}$ and
a unitary $U_{\pi} \in M(\cl{K}(H_{\pi}) \otimes C(\G))$
which satisfies $(\id\otimes\Delta)(U_{\pi}) = U_{\pi,12}U_{\pi,13}$.
The category of unitary representations of $\G$ is denoted by $\rep\G$,
and the full subcategory consisting of the finite dimensinal
representations is denoted by $\rep^{\fin} \G$. It is known
that this is semisimple and naturally made into a rigid \cstar-tensor
category with a simple unit. We fix a complete set of mutually
inequivalent irreducible representations of $\G$, which is
denoted by $\irr\G$.

Let $\pi$ be a finite dimensional unitary representation of $\G$ and
$\xi,\,\eta$ be elements of $H_{\pi}$.
Then $(\xi^*\otimes 1)U_{\pi}(\eta\tensor 1)$ defines an element of
$C(\G)$, called a \emph{matrix coefficient} of $\pi$.
The set of all matrix coefficients of all finite dimensional unitary
representations is
called the \emph{algebraic core} of $\G$ and denoted by $\cl{O}(\G)$.
We can make $\cl{O}(\G)$ into a Hopf \astar-algebra by using the product and
the coproduct of $C(\G)$. The counit and the antipode are denoted by
$\eps$ and $S$ respectively.

Let $\H$ be another compact quantum group. A \emph{morphism}
from $\G$ to $\H$ is a  homomorphism of Hopf \astar-algebras from
$\cl{O}(\H)$ to $\cl{O}(\G)$. One should note that
it need not be defined on $C(\H)$. Nonetheless this induces a
\cstar-tensor functor from $\rep^{\fin}\H$ to $\rep^{\fin}\G$
since $U_{\pi} \in B(H_{\pi})\tensor \cl{O}(\H)$
for any finite dimensional representation $\pi \in \rep^{\fin}\H$.

A \emph{quotient quantum group} of $\G$
is a pair of a compact quantum group $\H$ and a morphism
from $\G$ to $\H$ which is injective as a map from
$\cl{O}(\H)$ to $\cl{O}(\G)$. In this case the map extends to
an injective \astar-homomorphism from $C(\H)$ to $C(\G)$
since the Haar state on $\G$ restricts to the Haar state on $\H$.
On the other hand, a unital \cstar-subalgebra $B$ of $C(\G)$ satisfying
$\Delta(B) \subset B\tensor B$ and the cancellation property
gives rise to a quotient quantum group of $\G$.

For a quotient quantum group $\H$ of $\G$, the induced functor
from $\rep^{\fin}\H$ to $\rep^{\fin}\G$ is fully faithful and
has an image closed under tensor products, conjugate objects
and subobjects.
In particular it determines a subset of equivalence classes of
irreducible objects in $\rep^{\fin}\G$ containing the tensor unit
$\mathbf{1}_{\G}$ and closed under tensor products and conjugate
objects.
Conversely, such a subset $I$
defines a quotient quantum group of $\G$ by considering $B$,
the closed linear span of all matrix coefficients of $\pi \in I$.

As a conclusion, we have natural bijections among the following sets:

\begin{itemize}
 \item Isomorphism classes of quotient quantum groups of $\G$,
 \item Unital \cstar-subalgebras of $C(\G)$ closed under the coproduct
and satisfying the cancellation property,
 \item Subsets of $\irr\G$ closed under tensor products and conjugate objects.
\end{itemize}
\subsection{Discrete quantum groups and their morphisms} \label{subsec:dqg}
It is possible to formulate the notion of discrete quantum group in
the framework of locally compact quantum groups. See, for instance, \cite{MR2276175} for this approach.

To simplify the description, we treat discrete quantum group
simply as a dual notion of compact quantum group in this article.
Hence a \emph{discrete quantum group} $\Gamma$ is
a compact quantum group, denoted by $\hat{\Gamma}$,
and a \emph{morphism} from $\Gamma$ to another discrete quantum group $\Lambda$
is a morphism from $\hat{\Lambda}$ to $\hat{\Gamma}$.
Then we can give a natural definition of a \emph{quantum subgroup}
of $\Gamma$,
which is defined as a pair of a discrete quantum group $\Lambda$ and
a morphism from $\Lambda$ to $\Gamma$ such that $\hat{\Lambda}$ and the
corresponding morphism $\hat{\Gamma}\longrightarrow \hat{\Lambda}$
give a quotient of $\hat{\Gamma}$. This definition is equivalent
to Vaes's definition and Woronowicz's definition as shown in
\cite[Theorem 3.5, Theorem 6.2]{MR2980506} and \cite[Lemma 4.1]{MR4742819}.

In spite of our abstract approach, it is still possible to describe
\emph{function algebras} on a discrete quantum group.
By taking the dual of $\Delta$,
we introduce a product on the algebraic dual
$c(\Gamma) := \cl{O}(\hat{\Gamma})^{\vee}$, whose multiplicative unit is
the counit $\eps$ on the Hopf \astar-algebra $\cl{O}(\hat{\Gamma})$.
Moreover this algebra
is endowed with a natural \astar-structure defined by
$\omega^*(x) = \bar{\omega(S(x)^*)}$. Then the Peter-Weyl decomposition
$\cl{O}(\hat{\Gamma})\cong \bigoplus_{\pi \in \irr{\hat{\Gamma}}}B(H_{\pi})$ induces a \astar-isomorphism
$c(\Gamma)\cong \prod_{\pi \in \irr\hat{\Gamma}} B(H_{\pi})$.
Similarly we also have $c(\Gamma)\hat{\tensor}c(\Gamma) := (\cl{O}(\hat{\Gamma})\tensor \cl{O}(\hat{\Gamma}))^{\vee}\cong \prod_{\rho,\sigma \in \irr \hat{\Gamma}} B(H_{\rho})\tensor B(H_{\sigma})$.
Then the coproduct $\map{\Delta}{c(\Gamma)}{c(\Gamma)\hat{\tensor}c(\Gamma)}$, which is just the dual of the multiplication $\map{m}{\cl{O}(\G)\tensor \cl{O}(\G)}{\cl{O}(\G)}$, is presented as follows:
\begin{align*}
\text{the $(\rho,\sigma)$-component of } \Delta(X) &= \sum_{\pi \in \irr\hat{\Gamma}}\sum_{V \in \onb(\pi,\rho\tensor\sigma)}VX_{\pi}V^*,
\end{align*}
where the summation is defined as follows: take a complete orthogonal family $(V_i)_{i = 1}^k$ of isometries from $\pi$ to $\rho\tensor\sigma$
and consider $V_1X_{\pi}V_1^* + V_2X_{\pi}V_2^* + \cdots + V_nX_{\pi}V_n^*$.
By taking the uniformly bounded part, we obtain a
von Neumann algebra $\ell^{\infty}(\Gamma)$
and $\map{\Delta}{\ell^{\infty}(\Gamma)}{\ell^{\infty}(\Gamma)\bar{\tensor}\ell^{\infty}(\Gamma)}$. Note that any $x \in \cl{O}(\hat{\Gamma})$
defines a normal linear functional $\mathrm{ev}_x$ on $\ell^{\infty}(\Gamma)$.

Given a morphism $\map{\phi}{\Gamma'}{\Gamma}$ of discrete quantum
group, it induces a \astar-homomorphism $\map{\phi^*}{\ell^{\infty}(\Gamma)}{\ell^{\infty}(\Gamma')}$. When $\Gamma'$ is a quantum subgroup of $\Gamma$,
this homomorphism is the projection after regarding $\irr\hat{\Gamma'}$ as a subset of $\irr\hat{\Gamma}$:
\begin{align*}
 i^*\colon \ell^{\infty}(\Gamma)\longrightarrow \ell^{\infty}(\Gamma'),\quad (X_{\pi})_{\pi \in \irr\hat{\Gamma}}\longmapsto (X_{\pi})_{\pi \in \irr\hat{\Gamma'}}.
\end{align*}
We define $\ell^{\infty}(\Gamma/\Gamma')$ and $\ell^{\infty}(\gbag)$ as follows:
\begin{align*}
 \ell^{\infty}(\Gamma/\Gamma') &:= \{X \in \ell^{\infty}(\Gamma)\mid (\id\tensor i^*)\Delta(X) = X\tensor 1\} \\
& = \{X \in \ell^{\infty}(\Gamma) \mid (\id\tensor \mathrm{ev}_{x})\Delta(X)= \eps(x)X \text{ for all }x \in \cl{O}(\hat{\Gamma'})\},\\
 \ell^{\infty}(\Gamma'\backslash \Gamma) &:= \{X \in \ell^{\infty}(\Gamma)\mid (i^*\tensor \id)\Delta(X) = 1\tensor X\} \\
& = \{X \in \ell^{\infty}(\Gamma)\mid (\mathrm{ev}_{x}\tensor \id)\Delta(X)= \eps(x)X \text{ for all }x \in \cl{O}(\hat{\Gamma'})\}.
\end{align*}
These are also $\ell^{\infty}$-products of matrix algebras since
so is $\ell^{\infty}(\Gamma)$.
Note that these are anti-isomorphic by the antipode on $\ell^{\infty}(\Gamma)$.

\subsection{Hypergroup} \label{subsec:hypergroup}
See \cite{MR4696701} for fundamental facts on hypergroups.
\begin{defn}
 A \emph{hypergroup} is a triple $(H,\star,e)$ consisting of a set $H$,
a map $\map{\star}{H\times H}{\cl{P}(H)}$ and $e \in H$ with
the following conditions:
\begin{enumerate}
 \item For all $x,y,z \in H$, $(x\star y)\star z = x\star (y\star z)$.
 \item For all $x \in H$, $x\star e = \{x\}$.
 \item There exists a map $\map{\tend}{H}{H}$, called an inversion function, such that $y \in \bar{x}\star z$ and $x \in z\star\bar{y}$ hold when $z \in x\star y$.
\end{enumerate}
\end{defn}
We often omit $\star$ if there is no confusion.
\textbf{We use a notion of morphism between hypergroups, which is
different from that in literature} like \cite{MR4696701}.
A \emph{morphism} from $H$ to another hypergroup $H'$
is a map $\map{\phi}{H}{\cl{P}(H')}$ such that $f(e) = \{e\}$ and
$f(x\star y)= f(x)\star f(y)$.
This definition has
an advantage in the functorial construction of hypergroups
from discrete quantum groups.

Let $\Gamma$ be a discrete quantum group. We associate a hypergroup
$\Gamma_{\hyp}$ to $\Gamma$, whose underlying set is $\irr\hat{\Gamma}$, equipped with $\mathbf{1}$ as a neutral element.
The muliti-valued product is induced from the monoidal structure
on $\rep^{\fin}\hat{\Gamma}$:
\begin{align*}
 \pi\star\rho := \{\sigma \in \irr\hat{\Gamma}\mid \sigma < \pi\tensor \rho\},
\end{align*}
where $\sigma < \pi\tensor\rho$ means that $\sigma$ is contained in
$\pi\tensor \rho$ as a direct summand.
For any morphism $\map{\phi}{\Gamma'}{\Gamma}$, we associate
a morphism $\map{\phi_{\hyp}}{\Gamma'_{\hyp}}{\Gamma_{\hyp}}$ defined
as follows:
\begin{align*}
 \phi_{\hyp}(\pi') := \{\pi \in \Gamma_{\hyp}\mid \pi < \phi_*(\pi')\},
\end{align*}
where $\map{\phi_{\ast}}{\rep^{\fin}\hat{\Gamma'}}{\rep^{\fin}\hat{\Gamma}}$ is the induced \cstar-tensor functor.
In general we have $\abs{\phi_{\hyp}(\pi')} \neq 1$, hence it is natural to consider multi-valued maps as morphisms of hypergroups.
However, in the case that $\Gamma'$ is a quantum subgroup of $\Gamma$, the induced morphism from $\Gamma'_{\hyp}$ to $\Gamma_{\hyp}$ is single-valued and injective. Moreover its image forms a hypergroup with
the induced multiplication. 

A \emph{subhypergroup} of a hypergroup $H$
is a hypergroup whose underlying set is a subset of $H$ and
whose structure is also induced from $H$. In \cite{MR4696701},
a subhypergroup in our sense is called as a \emph{closed subset}.

The conclusion of Subsection \ref{subsec:cqg} implies
the following proposition.

\begin{prop} \label{prop:hypergroup description of quantum subgroups}
 Let $\Gamma$ be a discrete quantum group. Then there is a
canonical bijection between the following:
\begin{itemize}
 \item Isomorphism classes of quantum subgroups of $\Gamma$.
 \item Subhypergroups of $\Gamma_{\hyp}$.
\end{itemize}
\end{prop}

Let $G$ be a hypergroup and $H$ be a subhypergroup of $G$.
Then we have an equivalence relation $\sim_H$ on $G$ defined
as $x\sim_H y$ precisely when $x \in y\star H$. We call each
equivalence class a \emph{coset} and
the quotient set a \emph{coset space},
which is denoted by $G/H$. The coset containing $x \in G$ coincides
with $x\star H$. In a similar way $H\backslash G$ and $H\backslash G/H$
is also defined. Note that $H\backslash G/H$ has a natural hypergroup structure, which is induced from $G$.

Let $\Gamma'$ be a quantum subgroup of a discrete quantum group $\Gamma$. Note that
$\ell^{\infty}(\Gamma/\Gamma')\not\cong\ell^{\infty}(\Gamma_{\hyp}/\Gamma'_{\hyp})$ in general since $\ell^{\infty}(\Gamma/\Gamma')$ can be
non-commutative. Nonetheless it is still possible to relate
$\Gamma_{\hyp}/\Gamma'_{\hyp}$ to $\Gamma/\Gamma'$.
We regard $\ell^{\infty}(\Gamma_{\hyp}/\Gamma'_{\hyp})$ as a
von Neumann subalgebra of $\ell^{\infty}(\Gamma)$ by
$f\longmapsto (f(\pi\Gamma'_{\hyp})1_{\cl{H}_{\pi}})_{\pi \in \Gamma_{\hyp}}$. Similarly $\ell^{\infty}(\Gamma'_{\hyp}\backslash \Gamma_{\hyp})$
is considered as a von Neumann subalgebra of $\ell^{\infty}(\gbag)$

\begin{lemm} \label{lemm:comparison of quantum functions and functions}
The algebra $\ell^{\infty}(\Gamma_{\hyp}/\Gamma'_{\hyp})$ (resp. $\ell^{\infty}(\Gamma'_{\hyp}\backslash \Gamma_{\hyp})$) is contained in
$Z(\ell^{\infty}(\Gamma/\Gamma'))$ (resp. $Z(\ell^{\infty}(\gbag))$).
Moreover each direct summand $\delta_{\pi\Gamma'_{\hyp}}\ell^{\infty}(\Gamma/\Gamma')$ (resp. $\delta_{\Gamma'_{\hyp}\pi}\ell^{\infty}(\gbag)$)
is finite dimensional.
\end{lemm}
\begin{proof}
 The first statement follows from the definition of the coproduct on
$\ell^{\infty}(\Gamma)$. To show the second statement,
take $\pi \in \Gamma_{\hyp}$. Then we have a unital
\astar-representation of $\delta_{\pi\Gamma'_{\hyp}}\ell^{\infty}(\Gamma/\Gamma')$
defined by the projection to $B(\cl{H}_{\pi})$. It suffices to show
that this is faithful. Take $\rho \in \Gamma'_{\hyp}$.
Then $\pi\tensor \rho$ also defines a unital representation of
$\delta_{\pi\Gamma'_{\hyp}}\ell^{\infty}(\Gamma/\Gamma')$, whose direct summands
are precisely elements of $\pi\star\rho \subset \Gamma_{\hyp}$ by
definition of the coproduct on $\ell^{\infty}(\Gamma)$. On the other hand,
since $(\id\tensor \rho)\Delta(x) = x\tensor 1$ for $x \in \ell^{\infty}(\Gamma/\Gamma')$ and $\rho \in \Gamma'_{\hyp}$, we have
$(\pi\tensor\rho)\Delta(x) = 0$ if
$x \in \delta_{\pi\Gamma'_{\hyp}}\ell^{\infty}(\Gamma/\Gamma')$ and $\pi(x) = 0$.
Hence we have $\sigma(x) = 0$ for such $x$ and $\sigma \in \pi\star\rho$,
and the faithfulness of $\pi$ since $\pi\star\Gamma'_{\hyp}$ is the union of $(\pi\star\rho)_{\rho \in \Gamma'_{\hyp}}$.

The same argument works for $\gbag$ and $\Gamma'_{\hyp}\backslash \Gamma_{\hyp}$.
\end{proof}

\subsection{Right coideals and associated module categories}
For the definition and discussions on \cstar-tensor categories,
see \cite{MR3204665}. See also \cite{MR3121622} for
the definition and
fundamental notions on left module \cstar-categories over
\cstar-tensor categories.

Typical examples of left \cstar-module categories over a \cstar-tensor categories arise from right coideals of a compact quantum group $\G$, i.e.,
a \cstar-subalgebra $B \subset C(\G)$ such that
$\Delta(B) \subset B \tensor C(\G)$.
We can form a \cstar-category $\G\text{-}\mod_B$ of $\G$-equivariant Hilbert $B$-modules, which has a canonical structure of
a left $\rep^{\fin}\G$-module \cstar-category. See \cite{MR3675047} for the definition of $\G$-equivariant Hilbert modules and fundamental facts on them.

We say that a $\G$-equivariant Hilbert $B$-module $M$ is
\emph{finitely generated} if it is generated by a finite dimensional
$\G$-invariant subspace. It is known that such a module
has an embedding into $H_{\pi}\tensor B$ for some $\pi \in \rep^{\fin}\G$ as a direct summand. In particular $M$ is finitely generated projective as right $B$-module. The full subcategory consisting of
finitely generated $\G$-equivariant Hilbert $B$-module is denoted by $\G\text{-}\mod^{\fin}_B$.

We also have another way to construct a left $\rep^{\fin}\G$-module
\cstar-category from $B$. Using the algebraic core
$\cl{B} := B \cap \cl{O}(\G)$, we define a von Neumann subalgebra
$B^{\perp} \subset \ell^{\infty}(\hat{\G})$ as follows:
\begin{align*}
 B^{\perp} := \{X \in \ell^{\infty}(\hat{\G})\mid (\id\tensor \mathrm{ev}_b)\Delta(X) = \eps(b)X\}.
\end{align*}
Then this is a left coideal von Neumann subalgebra of $\ell^{\infty}(\hat{\G})$, i.e., it satisfies $\Delta(B^{\perp}) \subset \ell^{\infty}(\hat{\G})\bar{\tensor} B^{\perp}$.
Hence the \cstar-category of finite dimensional
normal \astar-representations of $B^{\perp}$, denoted by
$\rep^{\fin} B^{\perp}$, has a natural structure of left $\rep^{\fin}\G$-module \cstar-category since $\rep^{\fin}\G\cong \rep^{\fin} \ell^{\infty}(\hat{\G})$.
The following comparison was essentially pointed out by
K. De Commer and J.R. Dzokou Talla in \cite{MR4776189}.
See also \cite[Theorem 1]{MR549940} for a similar statement
in the algebraic setting.

\begin{prop} \label{prop:dual picture of module category}
As a left $\rep^{\fin}\G$-module \cstar-category, $\rep^{\fin} B^{\perp} \cong \G\text{-}\mod_B^{\fin}$.
\end{prop}

This kind of comparison also holds for $\G$-equivariant correspondences.
A $\G$-equivariant correspondence over $C(\G)$ is a pair
of $\G$-equivariant Hilbert $C(\G)$-module $M$ and
a left $C(\G)$ action on $M$ such that $\alpha_M(xm) = \Delta(x)\alpha_M(m)$ holds for all $x \in C(\G)$ and $m \in M$.

The symbol $\G\text{-}\corr^{\rfin}_{C(\G)}$ denotes
the \cstar-tensor category of $\G$-equivariant
correspondences over $C(\G)$ which are
finitely generated as $\G$-equivariant Hilbert $C(\G)$-modules.
The symbol $\rep^{\fin}\cl{O}(\G)$ denotes the category of
finite dimensional \astar-representations of $\cl{O}(\G)$.
The following is a special case of
\cite[Theorem 5.2]{MR4776189}
\begin{prop} \label{prop:corr and rep}
There is
an equivalence $\rep^{\fin}\cl{O}(\G)\cong \G\text{-}\corr^{\rfin}_{C(\G)}$ of \cstar-tensor categories, which is given by the following functor:
\begin{align*}
 H \longmapsto H\tensor C(\G),\quad T\longmapsto T\tensor \id_{C(\G)}.
\end{align*}
where the left action of $C(\G)$ is the diagonal action and we consider
the natural right actions of $C(\G)$ and the natural right coaction of
$C(\G)$ on the RHS.
\end{prop}

\subsection{Index of conditional expectation}
\label{subsec:index}

In this subsection, we collect definitions and facts about indices of
conditional expectations, introduced in \cite{MR996807}.

Let $A$ be a unital $\C$-algebra and $B$ be a unital $\C$-subalgebra of $A$.
A conditional expectation from $A$ to $B$ is a projection
$\map{E}{A}{B}$ as $B$-bimodules.
We say that $E$ is
\emph{of finite index} if it admits a \emph{quasi-basis}, a pair of
two finite families
$(u_i)_{i = 1}^n, (v_i)_{i = 1}^n \subset A$ satisfying
\[
 a = \sum_{i = 1}^n u_iE(v_ia) = \sum_{i = 1}^n E(au_i)v_i
\]
for all $a\in A$. In this case we can define the \emph{index}
of $E$ by the formula $\Index{E} = \sum_{i = 1}^n u_iv_i$.
This is independent of the choice of a quasi-basis, and $\Index{E}$ is an element of $Z(A)^{\times}$.

When $A$ and $B$ are \cstar-algebras, we only consider
conditional expectations in the usual sense. In this case
it is known that we can take a quasi-basis so that $v_i = u_i^*$.
We also define the \emph{scalar index} of $E$ as follows:
\[
 \Index^s E = \min\{c > 0\mid cE - \id_A\text{ is completely positive.}\}. 
\]
We have $\Index^s E = \nor{\Index E}$. (c.f. \cite[Th\'{e}or\`{e}me 3.5]{MR945550})

Consider a right coideal $B$ of $C(\G)$, where $\G$
is a compact quantum group. A conditional expectation $\map{E}{C(\G)}{B}$
is said to be $\G$-equivariant if
$(E\tensor \id)\circ \Delta = \Delta\circ E$. There is at most
one $\G$-equivariant conditional expectation since such an expectation
automatically preserves the Haar state.
Also note that $C(\hat{\Gamma'}) \subset C(\hat{\Gamma})$
always admits a $\G$-equivariant conditional expectation for any
inclusion $\Gamma' \subset \Gamma$ of discrete quantum groups.

We can interpret $\Index E$ as the categorical dimension of
a suitable object in $\G\text{-}\corr^{\rfin}_{C(\G)}$.
Here we need the notion of \cstar-Frobenius algebra. See
\cite[Section 2]{MR3933035} and \cite{MR3308880} for this topic.

If $\map{E}{C(\G)}{B}$
is a $\G$-equivariant conditional expectation of finite index,
we can form a \cstar-Frobenius algebra $(A,m,u)$ in $\G\text{-}\corr_{C(\G)}$:
\begin{itemize}
 \item $A = C(\G)_E\tensor_B C(\G)$,
 \item $\map{m}{A\tensor A}{A},\quad (x\tensor y)\tensor (z\tensor w) \longrightarrow x\tensor E(yz)w$,
 \item $\map{u}{C(\G)}{A},\quad x\longrightarrow \sum_{i = 1}^n xu_i\tensor u_i^*$,
\end{itemize}
where $C(\G)_E$ is a $\G$-equivariant Hilbert $B$-module based on
$C(\G)$ equipped with $\ip{y,x}_B = E(y^*x)$, and $(u_i)_{i = 1}^n$
is a quasi-basis of $E$.

Note that $\Index E = u^*mm^*u(1)\in C(\G)^\G = \C1_{C(\G)}$.
Hence we obtain $\Index E = (\Index^s E)1_{C(\G)}$.
We can also see that any morphism from $A$ to $C(\G)$ is
a scalar multiple of $u^*$, which is given by $x\tensor y \longrightarrow xy$, since the $\G$-equivariance implies that the image of
$1\tensor 1 \in A$ has to be a scalar and the compatibility with the
left and right $C(\G)$-actions implies that such a morphism is determined
by the image of $1\tensor 1$.
Hence this \cstar-Frobenius algebra is irreducible, and
\cite[Theorem 2.9]{MR3933035} implies that it is isomorphic to a standard Q-system.
Therefore the categorical dimension of $A$ can be calculated as
$\nor{u^*mm^*u} = \Index^s E$.



\section{Index of quantum subgroup in discrete quantum group}
In this section, we prove that several finiteness conditions on
right coideals are equivalent. We also
prove some formulae which are well-known
in the group case.

\subsection{Finite index right coideals}

Let $\G$ be a compact quantum group.

\begin{defn}
Let $B$ be a right coideal of $C(\G)$.
We say that $B$ is \emph{of finite index} when $B$ admits a
$\G$-equivariant conditional expectation $E$ of finite index.
In this case we define the index of $B$ as $[C(\G): B] := \Index^s E$.
In the other case, we define $[C(\G): B] = \infty$.
\end{defn}

The following definition is inspired by \cite[Subsection 8.2]{MR2183279}.
\begin{defn}
 Let $\cl{C}$ be a \cstar-tensor category
and $\cl{M}$ be a connected semisimple left $\rep^{\fin}\G$-module
\cstar-category with only finitely many irreducible objects.
\begin{enumerate}
 \item A \emph{dimension function} on $\cl{M}$
is a map $\map{d}{\mathrm{Obj}(\cl{M})}{[0,\infty)}$ satisfying
$d(X\oplus Y) = d(X) + d(Y)$ for all $X, Y \in \cl{M}$.
We say that $d$ is \emph{normalized} if $\min_{X\neq 0} d(X) = 1$.
 \item A dimension function $d$ is said to be $\cl{C}$-equivariant
if there is $d_U \in [0,\infty)$ for any $U \in \cl{C}$
such that $d(U\tensor X) = d_Ud(X)$ holds for all $X$.
 \item We define the
\emph{normalized Frobenius-Perron dimension} of $\cl{M}$
as
\begin{align*}
\FPdim \cl{M} := \sum_{X \in \irr\cl{M}} d(X)^2
\end{align*}
if $\cl{M}$ admits a normalized $\cl{C}$-equivariant dimension function $d$.
\end{enumerate}
\end{defn}
\begin{rema}
As an application of theory of Frobenius-Perron eigenvalue, we can see
that there is at most one normalized $\cl{C}$-equivariant
dimension function on $\cl{M}$ since $\cl{M}$ is connected.
In particular $\FPdim \cl{M}$ is well-defined.
\end{rema}
\begin{exam} \label{exam:FPdim for right coideals}
Let $B$ be a right coideal of $C(\G)$ such that
$\G\text{-}\mod^{\fin}_B$ has only finitely many irreducible objects.
Since $\G\text{-}\mod^{\fin}_B \cong \rep^{\fin} B^{\perp}$
by Proposition \ref{prop:dual picture of module category},
the forgetful functor on $\rep^{\fin} B^{\perp}$
induces a normalized $\rep^{\fin}\G$-equivariant dimension function of $\G\text{-}\mod^{\fin}_B$.
Hence we can define its Frobenius-Perron dimension,
and actually we have
\begin{align*}
 \FPdim \G\text{-}\mod^{\fin}_B = \dim_{\C} B^{\perp}.
\end{align*}
\end{exam}

In the following it is shown that
several natural finiteness conditions on right coideals are equivalent.

\begin{prop} \label{prop:characterizations of finite index right coideals}
For a right coideal of $B$ of $C(\G)$,
the following conditions are equivalent:
\begin{enumerate}
 \item A right $\cl{B}$-module $\cl{O}(\G)$ is finitely generated.
 \item There is a conditional expectation $\map{\cl{E}}{\cl{O}(\G)}{\cl{B}}$ which is of finite index and preserves the right coactions of $\cl{O}(\G)$.
 \item The right coideal $B$ is of finite index.
 \item $\dim_{\C} B^{\perp} < \infty$.
 \item The category $\G\text{-}\mod^{\fin}_B$
has only finitely many irreducible objects.
\end{enumerate}
In this case, we have $[C(\G):B] = \dim_{\C}B^{\perp} = \FPdim \G\text{-}\mod_B$.
\end{prop}
\begin{proof}
The equivalence of (iv) and (v) follows from Proposition \ref{prop:dual picture of module category}.
See Example \ref{exam:FPdim for right coideals}
for the equality $\dim_{\C} B^{\perp} = \FPdim \G\text{-}\mod^{\fin}_B$.

The implication (ii) $\Longrightarrow$ (i) follows by definition.

Next we show that (i) implies (iv). Take $(x_i)_{i = 1}^n \subset \cl{O}(\G)$ which generates $\cl{O}(\G)$ as a right $\cl{B}$-module. By taking adjoint, we see that $(x_i^*)_{i = 1}^n$ generates $\cl{O}(\G)$ as
a left $\cl{B}$-module. Consider the following map:
\begin{align*}
 \map{\Phi}{B^{\perp}}{\C^n},\quad f \longmapsto (f(x_i^*))_{i = 1}^n
\end{align*}
It suffices to show that this map is injective. Assume $\Phi(f) = 0$.
Then, for any $b \in B$ and $x_i$, we have
\begin{align*}
 f(bx_i^*) = (\ev_b\bar{\tensor} \ev_{x_i^*})\Delta(f) = \eps(b)\ev_{x_i^*}(f) = \eps(b)f(x_i^*) = 0.
\end{align*}
Since $(x_i^*)_{i = 1}^n$ generates $\cl{O}(\G)$, this means $f = 0$.

Next we show (iii) $\Longrightarrow$ (ii). Set $r := \Index^s E \in (0,\infty)$. Then we have $x^*x \le rE(x^*x) \le rx^*x$ for any $x \in C(\G)$,
which shows that the natural map $C(\G) \longrightarrow C(\G)_E$
is a topological isomorphism.

Take an irreducible decomposition $C(\G)_E \cong \bar{\bigoplus}_{\lbd \in \Lambda} M_{\lbd}$ as a $\G$-equivariant Hilbert $B$-module.
Then it gives a decomposition $C(\G)_E\tensor_B C(\G)\cong \bar{\bigoplus_{\lbd \in \Lambda}} M_{\lbd}\tensor_B C(\G)$. Since
each $M_{\lbd}$ is finitely generated, $M_{\lbd}\tensor_B C(\G)$
is also finitely generated.

By our assumption, $m:x\tensor y \in C(\G)_E\tensor_B C(\G)\longmapsto xy \in C(\G)$ is bounded. Hence there uniquely exists $\eta_{\lbd} \in M_{\lbd}\tensor_B C(\G)$ such that $\ip{\eta_{\lbd},\xi}_{C(\G)} = m(\xi)$ for $\xi \in M_{\lbd} \tensor_B C(\G)$.
Moreover we can present $\eta_{\lbd}$ as a finite sum $\sum_{i} u_{\lbd,i}\tensor v_{\lbd,i}$ since $M_{\lbd}$ is finitely generated projective.

For this family $\{\eta_{\lbd}\}_{\lbd \in \Lambda}$, we have
\begin{align*}
 \sum_{\lbd \in F}\nor{\eta_{\lbd}}^2 = \nor{\sum_{\lbd \in F}\eta_{\lbd}}^2 = \nor{m|_{\bigoplus_{\lbd \in F}M_{\lbd}\tensor_B C(\G)}}^2
\end{align*}
for any finite subset $F \subset \Lambda$. Hence we also have
\begin{align*}
 \sum_{\lbd \in \Lambda}\nor{\eta_{\lbd}}^2 = \nor{m}^2 < \infty.
\end{align*}

For a finite subset $F \subset \Lambda$, we define $\map{T_F}{C(\G)_E}{C(\G)_E}$ as follows:
\begin{align*}
 T_F(x) = \sum_{\lbd \in F}\sum_i v_{\lbd,i}^*\ip{u_{\lbd,i},x}_B.
\end{align*}
Then we have
\begin{align*}
 \nor{T_F(x)}_B \le \nor{T_F(x)} = \nor{\ip{\sum_{\lbd \in F}\eta_{\lbd},x\tensor 1}_{C(\G)}} \le \nor{\sum_{\lbd \in F}\eta_{\lbd}}\nor{x}_B.
\end{align*}
Hence $\{T_F\}_F$ converges to an operator $T$ in the operator norm.
On the other hand, $T_F = \id$ on $\bigoplus_{\lbd \in F}M_{\lbd}$.
Hence we have $T = \id_{C(\G)_E}$, which means $T_F$ is invertible
for some $F$. Since the image of $T_F$ is contained in $\bigoplus_{\lbd \in F} M_{\lbd}$, we can see that $\Lambda = F$ and
$C(G)_E$ is finitely generated as a $\G$-equivariant Hilbert $B$-module.

Now we have an embedding of $C(\G)_E$
into $H_{\pi}\tensor B$ for some $\pi \in \rep^{\fin}\G$.
Let $i$ be the embedding and $p = i^*$ be its adjoint. Then this induces
maps between the algebraic cores $\cl{O}(\G)$ and $\cl{B}$.
Fix an orthonormal basis $(e_i)_{i = 1}^n$ of $H_{\pi}$
and set $\xi_i := p(e_i\tensor 1) \in \cl{O}(\G)$. Then we have
\begin{align*}
 \xi = \pi(\xi) &= p\left(\sum_{i = 1}^n (e_i\tensor 1)\ip{e_i\tensor 1,i(\xi)}_B\right)\\
&= \sum_{i = 1}^n p(e_i\tensor 1)\ip{p(e_i\tensor 1),\xi}_B = \sum_{i = 1}^n \xi_iE(\xi_i^*\xi)
\end{align*}
for all $\xi \in \cl{O}(\G)$. Hence $(\xi_i)_{i = 1}^n$
 is a quasi-basis for $E|_{\cl{O}(\G)}$.

At last we show (iv) $\Longrightarrow$ (iii). By Lemma \ref{lemm:S2-invariance}, the right action of $\cl{O}(\G)$ on $B^{\perp}$ is $S^2$-invariant
i.e. we have $f\triangleleft x = f\triangleleft S^2(x)$.
In particular $f(x) = (f\triangleleft x)(1) = (f\triangleleft S^2(x))(1) = f(S^2(x))$ for $f \in B^{\perp}$ and $x \in \cl{O}(\G)$.

On the other hand, \cite[Lemma 2.6]{MR2276175} (c.f. \cite[Theorem 3.9]{MR3128415}) says
\begin{align*}
 \cl{B} = \{x \in \cl{O}(\G)\mid (f\tensor \id)\Delta(x) = f(1)x \text{ for all }f \in B^{\perp}\}.
\end{align*}
Now we can see the $S^2$-invariance of $\cl{B}$ since
$(f\tensor \id)\Delta(S^2(x)) = S^2((f\circ S^2)\tensor \id)\Delta(x)$.

 Then \cite[Theorem 4.2]{MR3863479} implies there is
a $\G$-equivariant conditional expectation $\map{E}{C(\G)}{B}$.
Now consider an irreducible decomposition of $C(\G)_E$
as $\G$-equivariant Hilbert $B$-module.
Note that $\pi$-isotypical component
of $C(\G)_E$ is finite dimensional for any $\pi \in \irr\G$.
This implies any irreducible $\G$-equivariant Hilbert $B$-module
appears in an irreducible decomposition of $C(\G)_E$ with finite multiplicity. Hence (v), which is equivalent to (iv), implies that $C(\G)_E$ has
an irreducible decomposition into finitely many irreducible submodules.
Then the discussion in the proof of (iii) $\Longrightarrow$ (ii) shows
that $B$ is of finite index.

For the comparison of $[C(\G):B]$ and $\dim_{\C}B^{\perp}$,
see Proposition \ref{prop:comparison of dimension and index}.
\end{proof}
\begin{rema}
If $\G$ is of Kac type, we can take an average of a possibly non-equivariant conditional expectation to obtain an equivariant conditional expectation. Moreover this procedure transforms conditional expectations of finite index into equivariant conditional expectations of finite index.
This means the conditions in Proposition \ref{prop:characterizations of finite index right coideals} are also equivalent
to the existence of conditional expectaion onto $B$ of finite index under
the assumption that $\G$ is of Kac type.
\end{rema}

\subsection{Index of quantum subgroup and some formulae}

Now we define the index of quantum subgroup.
In the remaining of this paper, $\Gamma$ denotes a discrete quantum group.
\begin{defn}
Let $\Gamma'$ be a quantum subgroup of $\Gamma$. We say that
 $\Gamma'$ is of finite index if $C(\hat{\Gamma'})$ is of finite index as a right coideal of $C(\hat{\Gamma})$. In this case we define the index of $\Gamma'$ in $\Gamma$ as $[C(\hat{\Gamma}): C(\hat{\Gamma'})]$,
which is denoted by $[\Gamma:\Gamma']$. We also set $[\Gamma:\Gamma'] = \infty$ when $\Gamma'$ is not of finite index.
\end{defn}

As expected from Proposition \ref{prop:hypergroup description of quantum subgroups}, we can understand the finiteness as a condition on
the associated hypergroups.

\begin{prop} \label{prop:subhypergroup characterization}
 Let $\Gamma'$ be a quantum subgroup of $\Gamma$. Then
the following conditions are equivalent:
\begin{enumerate}
 \item $[\Gamma:\Gamma'] < \infty$.
 \item $\abs{\Gamma_{\hyp}/\Gamma'_{\hyp}} < \infty$.
\end{enumerate}
Moreover we have $\abs{\Gamma_{\hyp}/\Gamma'_{\hyp}} \le [\Gamma:\Gamma']$.
\end{prop}
\begin{proof}
The equivalence follows from Lemma \ref{lemm:comparison of quantum functions and functions} and (iii) $\iff$ (iv) in Proposition \ref{prop:characterizations of finite index right coideals}.
The last statement also follows from the same lemma.
\end{proof}
\begin{rema}
In general $[\Gamma:\Gamma'] \neq \abs{\Gamma_{\hyp}/\Gamma'_{\hyp}}$.
For a counterexample, consider $\Gamma = \hat{S_3}$,
$\Gamma' = \hat{\Z/2\Z} = \Z/2\Z$ and an inclusion
corresponding to the surjective homomorphism
$S_3\longrightarrow \Z/2\Z$.  In this case $[\Gamma:\Gamma'] = 3\neq 2 = \abs{\Gamma_{\hyp}/\Gamma'_{\hyp}}$.
\end{rema}

\begin{exam}
Let $\mathbf{1}$ be the trivial quantum subgroup of $\Gamma$,
whose associated right coideal is $\C$.
Then we have $\C^{\perp} = \ell^{\infty}(\Gamma)$,
which means $[\Gamma:\mathbf{1}] = \dim_{\C}\ell^{\infty}(\Gamma)$.
Hence this is finite if and only if $\Gamma$ is a finite quantum group.
Moreover we also have $[\Gamma:1] = \abs{\Gamma}$ in this case.
\end{exam}

Now we provide some formulae on the index.

\begin{prop}
Let $\Gamma'$ be a quantum subgroup of $\Gamma$ and $\Gamma''$ be
a quantum subgroup of $\Gamma'$. Then we have $[\Gamma:\Gamma''] = [\Gamma:\Gamma'][\Gamma':\Gamma'']$.
\end{prop}
\begin{proof}
Let $\map{E}{C(\hat{\Gamma})}{C(\hat{\Gamma'})}$ and $\map{E'}{C(\hat{\Gamma'})}{C(\hat{\Gamma''})}$ be the equivariant conditional expectations. Then $\map{E'\circ E}{C(\hat{\Gamma})}[C(\hat{\Gamma''})]$ is also
the equivariant conditional expectation. Now the assertion follows from
a general statement on index \cite[Proposition 1.7.1]{MR996807} under the assumption
$[\Gamma:\Gamma'][\Gamma':\Gamma''] < \infty$. Next we assume
$[\Gamma:\Gamma''] < \infty$. Then a general statement on index shows
$[\Gamma':\Gamma''] < \infty$. We also have $[\Gamma:\Gamma'] < \infty$
as a corollary of (iii) $\iff$ (iv) in Proposition \ref{prop:characterizations of finite index right coideals} since $C(\hat{\Gamma'})^{\perp} \subset C(\hat{\Gamma''})^{\perp}$.
\end{proof}
For quantum subgroups of $\Gamma_1$ and $\Gamma_2$, we define
their intersection $\Gamma_1\cap\Gamma_2$ by
$C(\hat{\Gamma_1\cap \Gamma_2}) = C(\hat{\Gamma_1})\cap C(\hat{\Gamma_2})$. The corresponding hypergroup is $(\Gamma_1\cap \Gamma_2)_{\hyp} = (\Gamma_1)_{\hyp}\cap (\Gamma_2)_{\hyp}$.
\begin{prop}
Let $\Gamma_1$ and $\Gamma_2$ be quantum subgroups of $\Gamma$.
Then we have the following inequalities:
\begin{enumerate}
 \item $[\Gamma:\Gamma_1\cap \Gamma_2] \le [\Gamma:\Gamma_1][\Gamma:\Gamma_2]$.
 \item $[\Gamma_1:\Gamma_1\cap  \Gamma_2] \le [\Gamma:\Gamma_2]$.
\end{enumerate}
\end{prop}
\begin{proof}
(i) It suffices to show the statement under the assumption $[\Gamma:\Gamma_i] < \infty$ for $i = 1,2$. Let $E,E_1,E_2$ be the equivariant conditional expectations onto $C(\hat{\Gamma_1\cap \Gamma_2}), C(\hat{\Gamma_1})$
and $C(\hat{\Gamma_2})$ respectively. Then $(\Gamma_1\cap \Gamma_2)_{\hyp} = (\Gamma_1)_{\hyp}\cap (\Gamma_2)_{\hyp}$ implies
$\Delta\circ E = (E_1\tensor E_2)\circ\Delta$.
Hence $(\Index^sE_1)(\Index^sE_2)\Delta\circ E - \Delta$ is completely positive. Then the faithfulness of $\Delta$ implies
$\Index^s E \le (\Index^sE_1)(\Index^sE_2)$, which is nothing but
the statement.

(ii) This follows from that $E_2|_{C(\hat{\Gamma_1})}$ is
the equivariant conditional expectation onto $C(\hat{\Gamma_1\cap \Gamma_2})$,
where $E_2$ is the equivariant conditional expectation onto $C(\hat{\Gamma_2})$.
\end{proof}

The notion of normality on quantum subgroups of compact quantum group
was introduced by S. Wang \cite[Section 2]{MR1316765}
and investigated in \cite{MR3119236}
with a precise connection
to the algebraic theory \cite{MR1048073}. In the dual picture,
we say that a quantum subgroup $\Gamma'$ of $\Gamma$ is normal
when there is a normal quantum subgroup $\H$ of $\hat{\Gamma}$ such that
$C(\hat{\Gamma'}) = C(\H\backslash\hat{\Gamma})$. In this case $\hat{\H}$
is denoted by $\Gamma/\Gamma'$.

The following proposition is a direct consequence of Proposition \ref{prop:characterizations of finite index right coideals} combining
with $\hat{\Gamma}\text{-}\mod_{C(\H\backslash\hat{\Gamma})} \cong \rep^{\fin} \H$, which is true for any quantum subgroup of $\hat{\Gamma}.$

\begin{prop}
 Let $\Gamma'$ be a normal quantum subgroup of $\Gamma$ and $\Gamma'$ be its kernel. Then we have $[\Gamma:\Gamma'] = \abs{\Gamma/\Gamma'}$.
\end{prop}

\section{Descent of finite index quantum subgroup}
\subsection{Unimodular discrete quantum groups}
Let $\mathrm{\G}$ be a compact quantum group. We say that
$\G$ is \emph{of Kac type} when its Haar state is tracial.
This is equivalent to each of the following conditions on the antipode
$S$ of $\cl{O}(\G)$.
\begin{itemize}
 \item The square $S^2$ is the identity.
 \item The antipode $S$ is \astar-preserving.
\end{itemize}
We also say that a discrete quantum group $\Gamma$ is \emph{unimodular}
when the dual compact quantum group $\hat{\Gamma}$ is of Kac type.
In general, a discrete quantum group is not unimodular. For example,
the discrete dual of $SU_q(2)$ for $0 < q < 1$ is not unimodular.

Next we give a brief review of the maximal Kac quantum subgroup,
which was introduced in \cite[Appendix A]{MR2210362}. In this
section we follow \cite[Subsection 2.3]{MR3556413}.
See also
\cite{MR3285870} and \cite[Section 2]{MR4627100}.

Let $\G$ be a compact quantum group and $I$ be
an ideal of $\cl{O}(\G)$ generated by $\{S^2(a) - a\}_{a \in \cl{O}(\G)}$.
Since we have
\begin{align*}
 (S^2(a) -a)^* &= -a^* + S^{-2}(a^*) = S^2(-S^{-2}(a^*)) - (-S^{-2}(a^*)) \in I,\\
\Delta(S^2(a) - a) &= (S^2(a_{(1)}) - a_{(1)})\tensor S(a_{(2)}) + a_{(1)}\tensor (S^2(a_{(2)}) - a_{(2)})\\
& \in I\tensor \cl{O}(\G) + \cl{O}(\G)\tensor I,
\end{align*}
$\cl{O}(\G)/I$ has the induced structure of Hopf \astar-algebra.
Moreover it is generated by matrix coefficients of finite
dimensional unitary corepresentaitons in the sense in \cite[Theorem 1.6.7]{MR3204665}.
Hence there exists a compact quantum group $\G_{\mathrm{Kac}}$
such that $\cl{O}(\G_{\mathrm{Kac}})\cong \cl{O}(G)/I$. The morphism
of compact quantum group corresponding to
$\cl{O}(\G)\longrightarrow \cl{O}(\G_{\mathrm{Kac}})$ is denoted by
$\map{i}{\G_{\mathrm{Kac}}}{\G}$.

Since any morphism of compact quantum group preserves
the antipodes on the algebraic cores, $\G_{\mathrm{Kac}}$ has
the following universal property: For any morphism $\map{\phi}{\K}{\G}$ from a compact quantum group $\K$ of Kac type, there is a unique
morphism $\map{\phi_{\mathrm{Kac}}}{\K}{\G_{\Kac}}$ such that
$\phi = i\circ \phi_{\mathrm{Kac}}$.

By taking the dual, we obtain the \emph{unimodularization} of
a discrete quantum group $\Gamma$.
Namely we define the unimodularization $\Gamma_{\mathrm{uni}}$
as $\hat{\hat{\Gamma}_{\mathrm{Kac}}}$ with $\map{q}{\Gamma}{\Gamma_{\uni}}$, which is the dual of $i_{\Kac}$.
Again $(\Gamma_{\uni},q)$ has a universal property: Any morphism from $\Gamma$
to a unimodular discrete quantum group $\Lambda$ uniquely
factors through $\map{q}{\Gamma}{\Gamma_{\uni}}$.

The following proposition is a corollary of \cite[Appendix A]{MR2210362}, which is applied to the universal form of $\G$.

\begin{prop} \label{prop:descent of fin rep}
Let $\G$ be a compact quantum group. Then the tensor category
$\rep^{\fin}\cl{O}(\G)$ of finite dimensional \astar-representations
of $\cl{O}(\G)$ is equivalent to $\rep^{\fin}\cl{O}(\G_{\Kac})$
under the functor induced by $i$.
\end{prop}

Let $(\pi,\cl{H})$ be a finite dimensional \astar-representations
of $\cl{O}(\G)$. If $\G$ is of Kac type,
$\bar{\pi}(x)\xi := \bar{S(x^*)\xi}$ defines a finite dimensional
\astar-representation of $\cl{O}(\G)$. Then the following maps are
intertwiners and gives a solution of the conjugate equation in $\rep^{\fin}\cl{O}(\G)$:
\begin{align*}
 \map{R^*}{\bar{\pi}\tensor \pi}{\eps};\, \bar{\xi}\tensor \eta\longmapsto \ip{\xi,\eta},\\
 \map{\bar{R}^*}{\pi\tensor \bar{\pi}}{\eps};\, \xi\tensor \bar{\eta}\longmapsto \ip{\eta,\xi}.
\end{align*}
Hence $\rep^{\fin}\cl{O}(\G)$ is rigid and the intrinsic dimension of $(\pi,\cl{H})$ is equal to or smaller than $\dim_{\C} \cl{H}$.
On the other hand, the intrinsic dimension is equal to or greater than
$\dim_{\C}\cl{H}$ since the forgetful functor $\rep^{\fin}\cl{O}(\G)\longrightarrow \hilb^{\fin}$ is unitary and monoidal. Hence the intrinsic
dimension of $(\pi,\cl{H})$ is $\dim_{\C}\cl{H}$.

The following is a combination of this observation and Proposition \ref{prop:descent of fin rep}.

\begin{prop}[c.f. {\cite[Corollary 6.6]{MR3285870}}] \label{prop:rigidity and dimension}
Let $\G$ be a quantum group. Then $\rep^{\fin}\cl{O}(\G)$ is rigid
and the forgetful functor to $\hilb^{\fin}$ is dimension-preserving.
\end{prop}

Now we prove the remaining part of Proposition \ref{prop:characterizations of finite index right coideals}.

\begin{lemm} \label{lemm:S2-invariance}
Let $A$ be a finite dimensional right $\cl{O}(\G)$-module \cstar-algebra.
Then the action of $\cl{O}(\G)$ is $S^2$-invariant
in the sense that $a \triangleleft x = a\triangleleft S^2(x)$
holds for any $a \in A$ and $x \in \cl{O}(\G)$.
\end{lemm}
\begin{proof}
By Vaes's implementation theorem \cite[Proposition 3.7, Theorem 4.4]{MR1814995},
there is a \astar-representation $\pi$
of $\cl{O}(\G)$ on $L^2(A)$ such that $a \triangleleft x = \pi(S(x_{(1)}))a\pi(x_{(2)})$ on $L^2(A)$. Then Proposition \ref{prop:descent of fin rep} implies
$\pi(S^2(x)) = \pi(x)$ for all $x \in \cl{O}(\G)$ and we have
\begin{align*}
 a\triangleleft S^2(x) = \pi(S^3(x_{(1)}))a\pi(S^2(x_{(1)})) = \pi(S(x_{(1)}))a\pi(x_{(2)}) = a\triangleleft x.
\end{align*}
\end{proof}
\begin{prop} \label{prop:comparison of dimension and index}
 Let $B$ be a finite index right coideal of $C(\G)$.
Then we have $[C(\G):B] = \dim_{\C}B^{\perp}$.
\end{prop}
\begin{proof}
As discussed in Subsection \ref{subsec:index},
the index is equal to the categorical dimension of
$C(\G)_E\tensor_B C(\G)$ in $\G\text{-}\corr^{\fin}_{C(\G)}$.
Moreover Proposition \ref{prop:corr and rep} implies that
there is a finite dimensional \astar-representation $H$ of
$\cl{O}(\G)$ such that $C(\G)_E\tensor_B C(\G) \cong H\tensor C(\G)$.
Then Proposition \ref{prop:rigidity and dimension} implies
that the categorical dimension of $C(\G)_E\tensor_B C(\G)$ is equal
to $\dim_{\C} H$.
Since $\G$-invariant part of $H\tensor C(\G)$ is
$H\tensor \C1_{\C(\G)}$,
it suffices to show $\dim_{\C} (C(\G)_E\tensor_B C(\G))^{\G} = \dim_{\C} B^{\perp}$.

Note that the image of $\cl{O}(\G)\tensor \cl{O}(\G)$ in
$C(\G)_E\tensor_B C(\G)$ can be identified with the subalgebra of $B(L^2(\G))$ generated by $\cl{O}(\G)$ and the orthogonal projection $e_B$ onto
the closure of $B$ in $L^2(\G)$. Then \cite[Lemma 3.6, Lemma 3.8]{MR2335776} implies the statement.
\end{proof}

\subsection{Descent to the canonical unimodular quotient}

Next we show an imprimitivity-type result on finite index
right coideals. Let $\map{q}{\cl{O}(\G)}{\cl{O}(\G_{\Kac})}$
be the \astar-homomorphism corresponding to $i$.

\begin{prop} \label{prop:imprimitivity}
Let $B \subset C(\G)$ be a finite index right coideal of $C(\G)$.
Then there is a finite index right coideal $B_{\Kac}$ of $C(\G_{\Kac})$
such that $\cl{B} = q^{-1}(\cl{B}_{\Kac})$,
where $\cl{B}$ and $\cl{B}_{\Kac}$ are the algebraic cores of $B$
and $B_{\Kac}$ respectively. Moreover we have $[C(\G):B] = [C(\G_{\mathrm{Kac}}):B_{\mathrm{Kac}}]$.
\end{prop}
\begin{proof}
 By assumption $B^{\perp}$ is finite dimesional. Then
Lemma \ref{lemm:S2-invariance} implies
$f\triangleleft x = f\triangleleft S^2(x)$ for all $f \in B^{\perp}$
and $x \in \cl{O}(\G)$.
Hence the action descends to an action of $\cl{O}(\G_{\mathrm{Kac}})$.
Now we define $\cl{B}_{\mathrm{Kac}} \subset \cl{O}(\G_{\mathrm{Kac}})$
as follows:
\begin{align*}
 \cl{B}_{\mathrm{Kac}} := \{x \in \cl{O}(\G_{\mathrm{Kac}})\mid (f\triangleleft x_{(1)})(1)x_{(2)} = f(1)x \text{ for all }f \in B^{\perp}\}.
\end{align*}
Then $q^{-1}(\cl{B}_{\mathrm{Kac}}) = \cl{B}$ holds
since we have $(f\triangleleft q(x))(1) = f(x)$
and we also have
\begin{align*}
 \cl{B} = \{x \in \cl{O}(\G)\mid (f\tensor \id)\Delta(x) = f(1)x \text{ for all }f \in B^{\perp}\}
\end{align*}
by \cite[Lemma 2.6]{MR2276175} (c.f. \cite[Theorem 3.9]{MR3128415}).
This fact implies that the closure $B_{\mathrm{Kac}}$ of
$\cl{B}_{\mathrm{Kac}}$ is a right coideal of $C(\G_{\mathrm{Kac}})$
since $\map{q}{\cl{O}(\G)}{\cl{O}(\G_{\mathrm{Kac}})}$ is surjective.
We also can see that $B_{\mathrm{Kac}}$ is of finite index
and $[C(\G):B] = [C(\G_{\mathrm{Kac}}):B_{\mathrm{Kac}}]$
since $B^{\perp} = B_{\mathrm{Kac}}^{\perp} \subset \ell^{\infty}(\hat{\G_{\Kac}}) \subset \ell^{\infty}(\hat{\G})$.
\end{proof}

We apply this proposition to
quantum subgroups of finite index.
Let $\Gamma$ be a discrete quantum group and $\Gamma' \subset \Gamma$
be a quantum subgroup of finite index. Then the above theorem implies
that there is a coideal $B_{\Kac}$ of $C(\hat{\Gamma_{\uni}})$ such
that $\cl{O}(\hat{\Gamma'}) = q^{-1}(\cl{B}_{\Kac})$. Then the
surjectivity of $q$ implies that $\cl{B}_{\Kac} = q(\cl{O}(\hat{\Gamma'}))$. Hence $B_{\Kac}$ also corresponds to a finite index quantum subgroup $\Gamma'_{\uni}$ of $\Gamma'$.

Before stating a special property of $\Gamma'$, we prepare
a notion on multi-valued maps. Let $X,Y$
be sets and consider a map $\map{\phi}{X}{\cl{P}(Y)}$. We define
$\phi^{-1}(A) := \{x \in X\mid A\cap \phi(x) \neq \emptyset\}$ for
a subset $A \subset Y$.
We say that a subset $A \subset X$ is $\phi$-stable when
$\phi^{-1}(\phi(A)) = A$ holds.
\begin{lemm} \label{lemm:stability}
If $\Gamma'$ is a quantum subgroup of $\Gamma$
of finite index, $\Gamma'_{\hyp}$ is $q_{\hyp}$-stable.
\end{lemm}
\begin{proof}
We show that $q_{\hyp}(\Gamma'_{\hyp}) = (\Gamma'_{\uni})_{\hyp}$
and $q_{\hyp}^{-1}((\Gamma'_{\uni})_{\hyp}) = \Gamma'_{\hyp}$.
The former equality follows from the description of $q_{\hyp}$
using the Peter-Weyl decomposition.

To see the latter equality, take $\pi \in q_{\hyp}^{-1}((\Gamma'_{\uni})_{\hyp})$. Then the image of some matrix coefficients of $\pi$ is contained
in $q(\cl{O}(\hat{\Gamma'}))$. Hence $B(\cl{H}_{\pi}) \cap q^{-1}(\cl{O}(\hat{\Gamma'_{\uni}})) \neq 0$. On the other hand we have the following
dichotomy for $\cl{O}(\hat{\Gamma'}) = q^{-1}(\cl{O}(\hat{\Gamma'_{\uni}}))$: either of
$B(\cl{H}_{\rho}) \cap \cl{O}(\hat{\Gamma'}) = 0$ or $B(\cl{H}_{\rho}) \subset \cl{O}(\hat{\Gamma'})$ holds for any $\rho \in \irr\hat{\Gamma}$.
Hence $B(\cl{H}_{\pi}) \subset \cl{O}(\hat{\Gamma'})$, which means $\pi \in \Gamma'_{\hyp}$.
\end{proof}

The following lemma can be seen from the definition of $\phi$-stability.

\begin{lemm} \label{lemm:equivalence relation}
Let $X,Y$ be sets and consider $\map{\phi}{X}{\cl{P}(Y)}$.
Let $\sim_{\phi}$ be an equivalence relation on $X$ generated by $\{(x,x')\in X\times X\mid \phi(x)\cap \phi(x') \neq \emptyset\}$.
Then $A \subset X$
is $\phi$-stable if and only if $A$ is a union of equivalence classes
of $\sim_{\phi}$.
\end{lemm}

This lemma has the following corollary.

\begin{lemm}
Let $\map{\phi}{G}{H}$ be a morphism of hypergroups.
There is a minimum $\phi$-stable subhypergroup of $G$.
\end{lemm}
\begin{proof}
We can obtain the minimum by taking the intersection over all $\phi$-stable subhypergroups, which is actually a subhypergroup and $\phi$-stable
from the previous lemma.
\end{proof}

\begin{defn}
 A \emph{stable kernel} of a hypergroup morphism
$\map{\phi}{G}{H}$ is the minimum $\phi$-stable hypersubgroup.
It is denoted by $\stker \phi$
\end{defn}
\begin{rema}
If $\phi$ is a group homomorphism, the $\phi$-stable kernel coincides
with the usual kernel.
\end{rema}

The following is a combination of Lemma \ref{lemm:stability}
and a property of double coset hypergroups (\cite[Theorem 3.4.6]{MR4696701}).

\begin{thrm} \label{thrm:reduction to the unimodularization}
Let $\Gamma$ be a discrete quantum group and $\map{q}{\Gamma}{\Gamma_{\uni}}$ be its unimodularization. Then there is a natural 1-to-1
correspondence between the following:
\begin{itemize}
 \item finite index quantum subgroups of $\Gamma$,
 \item finite index subhypergroups of $\stker q_{\hyp}\backslash \Gamma_{\hyp}/\stker q_{\hyp}$.
\end{itemize}
\end{thrm}

\begin{exam}[{\cite[Theorem 5.10]{freslon2024discretequantumsubgroupsfree}, \cite[Main Theorem]{MR4824928}}]
Take $Q \in GL_N(\C)$ with $N \ge 2$. S. Wang and A. van Daele introduced a free unitary quantum group $U_Q^+$ in their paper \cite{MR1382726}. The representation theory of $U_Q^+$ is investigated in
\cite{MR1484551}, which shows that $\hat{U_Q^+}_{\hyp}$ is independent
 of $Q$ and $N \ge 2$. To describe this hypergroup and determine
finite index quantum subgroups of $\hat{U_Q^+}$, we only consider
the case of $Q = \mathrm{diag}\{q,q^{-1}\}$ where $q \in (0,1)$.
Proposition \ref{prop:subhypergroup characterization} implies
that the other cases follow from this case.

The algebra $C(U_Q^+)$ is generated by $u_{ij}$ ($1 \le i,j \le 2$)
with the following relations:
\begin{align*}
 U^* = U^{-1},\quad (Q\bar{U}Q^{-1})^* = (Q\bar{U}Q^{-1})^{-1}
\end{align*}
where $U = (u_{ij})_{ij}$ and $\bar{U} = (u_{ij}^*)_{ij}$.
The coproduct on $C(U_Q^+)$ is uniquely determined by
\begin{align*}
 \Delta(u_{ij}) = u_{i1}\tensor u_{1j} + u_{i2}\tensor u_{2j},
\end{align*}
which makes $U_Q^+ = (C(U_Q^+),\Delta)$ into a compact quantum group.
The algebraic core is the \astar-subalgebra generated by $\{u_{ij}\}_{ij}$ and the square of the antipode $S$ is given as follows:
\begin{align*}
 S^2(u_{11}) = u_{11},\quad
 S^2(u_{12}) = q^4u_{12},\quad
 S^2(u_{21}) = q^{-4}u_{21},\quad
 S^2(u_{22}) = u_{22}.
\end{align*}
Hence the maximal Kac quantum subgroup of $U_Q^+$ can be identified
with $\hat{F_2}$, the dual of the free group of rank 2.h
The corresponding \astar-homomorphism
$\map{\phi}{\cl{O}(U_Q^+)}{\cl{O}(\hat{F_2}) = \C[F_2]}$
is given by $\phi(u_{ii}) = s_i$ and
$\phi(u_{ij}) = 0$ when $i\neq j$, where $s_1$ and $s_2$
are the standard generators of $F_2$.

Note that $U$ and $Q\bar{U}Q^{-1}$ define irreducible finite
dimensional representation $r_1$ and $r_{-1}$ respectively.
Moreover we can see that $r_{-1}$ is the conjugate representation of
$r_1$. To describe all irreducible representations, it is convenient
to consider the underlying set $S$
of the free product $\Z_{\ge 0}\ast\Z_{\le 0}$ of monoid i.e.
the set of finite words consisting of $[1]$ and $[-1]$.
Then we have $\hat{U_Q^+}_{\hyp} \cong S$
\textbf{as a set}. The multi-valued multiplication and
the inverse function can be presented as follows:
\begin{align*}
 ([n_1][n_2]\cdots [n_k])^{-1} &= [-n_k][-n_{k - 1}]\cdots [-n_1],\\
 x\star y &= \{x'y'\mid x = x'w, y = w^{-1}y' \text{ for some }w \in S\}
\end{align*}
where $xy$ is the product of $x,y \in S$ with respect
to the multiplication of $\Z_{\ge 0}\ast \Z_{\le 0}$.
More strongly we have
\begin{align*}
 \pi_x\tensor \pi_y = \bigoplus_{z \in x\star y} \pi_z
\end{align*}
where $\pi_x$ is the corresponding irreducible representation of $U_Q^+$.

To describe the stable kernel of $\phi:\hat{U_Q^+}_{\hyp}\longrightarrow F_2$, we introduce some notions. We say that $w' \in S$ is a \emph{subword} of $w' \in S$ when we can obtain $w'$ by removing $[1][-1]$
or $[-1][1]$ from $w$ repeatedly.
For $w = [n_1][n_2]\cdots [n_l] \in S$, $g \in F_2$
is said to be \emph{of $w$-type} if the reduced expression of $g$ is of the form $g = s_{i_1}^{n_1}s_{i_2}^{n_2}\cdots s_{i_l}^{n_l}$.
We also say $g$ is \emph{of sub-$w$-type} if it is of $w'$-type
for some subword $w'$ of $w$.

The image of $\pi_w$ is the subset of elements of sub-$w$-type.
This property can be firstly seen for an alternative word i.e.
a word of the form $[1][-1]\cdots$ or $[-1][1]\cdots$.
Actually we can use the fusion rule to prove the property
by induction on the length of the word.T
Take a general element $w = [n_1][n_2]\cdots [n_l]$ and
consider $\{i \mid n_in_{i + 1} = 1\} = \{i_1 < i_2 < \cdots < i_k\}$.
We also set $i_0 = 0$.
Then $w_p = [n_{i_{p - 1} + 1}][n_{i_{p - 1} + 2}]\cdots [n_{i_p}]$
is alternative and we have $w = w_1w_2\cdots w_p$. Moreover
we also have $\pi_w = \pi_{w_1}\tensor \pi_{w_2}\tensor \cdots\tensor \pi_{w_p}$. Hence the image of $w$ is the product of the images of $w_1, w_2,\dots, w_p$ in this order. Since any subword of $w$
appears as a product of subwords of $w_1,w_2,\cdots, w_p$,
we obtain the assertion.

In general we say that \emph{$w \in S$ reduces to $n \in \Z$}
if $w = [n_1][n_2]\cdots [n_l]$ and $n_1 + n_2 + \cdots + n_l = n$.
The set of words in $S$ reducing to $n \in \Z$ is denoted by $W_n$.
Then $w$ has $e$ as its subword if and only if $w$ reduces to $0$,
hence $W_0 \subset \stker \phi$. On the other hand,
It is not difficult to see that $W_0$ defines
a subhypergroup of $\hat{U_Q^+}$. Hence we have $\stker\phi = W_0$.
Moreover we also have each $W_n$ is a double coset of $W_0$
since $w^{-1}\star w' \subset W_0$ if $w$ and $w'$ reduces to
the same integer. Hence $\stker\phi \backslash \hat{U_Q^+}_{\hyp}/\stker\phi \cong \Z$ as a hypergroup. Now Theorem \ref{thrm:reduction to the unimodularization} concludes that
the following subhypergroups
exhaust all subhypergroups corresponding to finite index quantum groups of $\hat{U_Q^+}$: for a positive integer $k$, we have
\begin{align*}
 H_k = \bigcup_{n\in \Z} W_{kn}.
\end{align*}
This corresponds to $W^{(k)}$ in \cite[Definition 5.4]{freslon2024discretequantumsubgroupsfree} and $(k\Z,\Z)$ in the classification
theorem \cite[Main Theorem]{MR4824928}.
\end{exam}

\begin{exam}[{\cite[Theorem 5.4]{MR4728596}}]
\label{exam:dual of cLg}
Let $K$ be a connected simply-connected compact Lie group.
In this case a quantum subgroup of $\hat{K}$
corresponds to a quotient compact group of $K$.
Moreover a finite index quantum subgroup of $\hat{K}$
corresponds to a quotient compact group of $K$ with the finite kernel.
Then general theory for compact Lie groups says that
such quotientss are classified by subgroups of $P/Q$,
where $P$ and $Q$ are the associated weight lattice and
the associated root lattice respectively.

We can recover this classification using the Drinfeld-Jimbo
deformation of $K$.
The Drinfeld-Jimbo deformation $K_q$ is defined as
a compact quantum group for each of
$q \in (0,1)$. See \cite[Section 2.4]{MR3204665} for detailed
description. We briefly recall some outstanding properties of $K_q$:
\begin{itemize}
 \item It contains the maximal torus $T$ of $K$ as the maximal Kac quantum subgroup. (\cite[Lemma 4.10]{MR2335776})
 \item We have $(\hat{{K}_q})_{\hyp}\cong \hat{{K}}_{\hyp}$ as hypergroups. Moreover we can take this isomorphism so that the unimodularization $\map{q_{\hyp}}{(\hat{{K}_q})_{\hyp}}{\hat{T}_{\hyp}}$ coincides with the morphism $\phi\colon \hat{K}_{\hyp}\longrightarrow{\hat{T}_{\hyp}}$ induced from $T \subset K$.
\end{itemize}
We determine $\stker \phi$ from these properties. Note that
$\hat{T}_{\hyp} = P$.
By the highest weight theory,
all irreducible representations corresponding to dominant
integral weights in $Q$ are contained in $\stker \phi$.
On the other hand, the set of such representations defines
a subhypergroup $\hat{K}_0$ of $\hat{K}_{\hyp}$. Hence we have
$\stker \phi = \hat{K}_0$. For $[\lbd] \in P/Q$,
we define $\hat{K}_{[\lbd]}$ as the set of
irreducible representations of $K$ whose weights are contained in
$[\lbd] \subset P$. Then, $[\lbd] \in P/Q \longmapsto \hat{K}_{[\lbd]}$
gives an isomorphism $P/Q\cong \hat{K}_{\hyp}/\stker\phi = \stker\phi\backslash\hat{K}_{\hyp}/\stker\phi$. Hence the finite index quantum subgroups of $\hat{K}$ are classified by the subgroups of $P/Q$.
\end{exam}

\section{Application to free products of discrete quantum groups}

In this section we apply Theorem \ref{thrm:reduction to the unimodularization} to free products of discrete quantum groups with a special property.
We begin with some discussion on free products of discrete quantum groups and hypergroups.

\subsection{Free products of discrete quantum groups and hypergroups}
Let $\HypGrp$ be a category of hypergroups, \textbf{whose
morphisms are defined in Subsection \ref{subsec:hypergroup}}.
Note that another category of hypergroups, which is equipped with
a different type of morphisms, has been extensively studied in the literature.
See \cite{nakamura2024categorieshypermagmashypergroupsrelated}
for instance.

 In the category $\HypGrp$,
we have a coproduct for any family of objects.
Let $\{H_{\lbd}\}_{\lbd \in \Lambda}$ bea family of hypergroups. Then the underlying set of the coproduct is given by:
\begin{align*}
 \bigast_{\lbd \in \Lambda} H_{\lbd} = \{e\}\sqcup\bigsqcup_{n = 1}^{\infty}\bigsqcup_{\lbd_1\neq\lbd_2\neq\cdots\neq \lbd_n}H_{\lbd_1}^*\times H_{\lbd_2}^*\times \cdots\times H_{\lbd_n}^*
\end{align*}
where $H^* := H\setminus\{e\}$ for a hypergroup $H$.
The following description of this set is also convenient:
define $\map{\iota}{\bigsqcup_{\lbd \in \Lambda} H_{\lbd}^*}{\Lambda}$ as $\iota(h) = \lbd$ for $h \in H_{\lbd}^*$. Let $W$ be the
set of finite words on $\bigsqcup_{\lbd \in \Lambda} H_{\lbd}^*$.
Then we can identify $\bigast_{\lbd \in \Lambda} H_{\lbd}$
with the following subset of $W$.
\begin{align*}
\{e\}\cup \{h_1h_2\cdots h_n \in W\mid \iota(h_1)\neq \iota(h_2)\neq\cdots \neq \iota(h_n)\}.
\end{align*}

Then a hypergroup structure on $\bigast_{\lbd \in \Lambda} H_{\lbd}$ is introduced as follows: The neutral element is $e$.
The inverse is given by $h_1h_2\cdots h_n \longmapsto \bar{h_n}\,\bar{h_{n - 1}}\cdots \bar{h_1}$. The multi-valued product
is defined recursively as follows:
\begin{align*}
& h_1h_2\cdots h_n\star h_{n + 1}h_{n + 2}\cdots h_{n + m} \\
&:=
\begin{cases}
 \{h_1h_2\cdots h_{n + m}\} & (\iota(h_n)\neq \iota(h_{n + 1})),\\
 \{h_1h_2\cdots h_{n - 1}hh_{n + 2}\cdots h_{n + m}\mid h \in h_nh_{n +1}\} & (\iota(h_n) = \iota(h_{n + 1}), h_n\neq \bar{h_{n + 1}})
\end{cases}
\end{align*}
and
\begin{align*}
h_1\cdots h_{n - 1}h_n\star \bar{h_n}h_{n + 2}\cdots h_{n + m} :=
 \{h_1\cdots &h_{n - 1}hh_{n + 2}\cdots h_{n + m}\mid h \in (h_n\star \bar{h_n})\setminus\{e\}\}\\
&\cup h_1\cdots h_{n-1}\star h_{n + 2}\cdots h_{n + m}.
\end{align*}
We call this hypergroup the \emph{free product} of $\{H_{\lbd}\}_{\lbd \in \Lambda}$. It is not difficult to see that $\bigast_{\lbd \in \Lambda} H_{\lbd}$ gives a coproduct of $\{H_{\lbd}\}_{\lbd \in \Lambda}$ in $\HypGrp$ with the canonical morphisms
$\{\map{i_{\lbd}}{H_{\lbd}}{H}\}_{\lbd \in \Lambda}$.

We also have the notion of free products of
discrete quantum groups, which was introduced in
\cite[Theorem 3.4]{MR1316765} as free products of
Woronowicz \cstar-algebras. The following fact is known:

\begin{prop}[{\cite[Theorem 3.10]{MR1316765}}]
 Let $(\Gamma_{\lbd})_{\lbd \in \Lambda}$ be a family of discrete
quantum groups. Then we have a canonical isomorphism
$(\bigast_{\lbd \in \Lambda} \Gamma_{\lbd})_{\hyp}\cong \bigast_{\lbd \in \Lambda} (\Gamma_{\lbd})_{\hyp}$ as hypergroups.
\end{prop}

We also put the identification on the unimodularization,
which immediately follows from the universal property of
the unimodularization and the free product \cite[Theorem 3.4]{MR1316765}.
\begin{prop}
Let $(\Gamma_{\lbd})_{\lbd}$ be a family of discrete quantum groups.
Then there is a canonical isomorphism $(\bigast_{\lbd \in \Lambda}\Gamma_{\lbd})_{\uni}\cong \bigast_{\lbd \in \Lambda}(\Gamma_{\lbd})_{\uni}$
\end{prop}

\subsection{Finite index quantum subgroups of free products}
Let $\{\Gamma_{\lbd}\}_{\lbd \in \Lambda}$ be a family of
discrete quantum groups and $\map{q_{\lbd}}{\Gamma_{\lbd}}{\Gamma_{\lbd,\uni}}$ be the unimodularization for each $\lbd \in \Lambda$.
Then $\map{\bigast_{\lbd \in \Lambda} q_{\lbd}}{\bigast_{\lbd \in \Lambda}\Gamma_{\lbd}}{\bigast_{\lbd \in \Lambda}\Gamma_{\lbd,\uni}}$
can be regarded as the unimodularization.

It seems difficult to determine $\stker (\bigast_{\lbd\in \Lambda} q_{\lbd})$ in general, but we can reduce the problem to determining
each $\stker q_{\lbd}$ in the specific case.

Let $G$ be a hypergroup and $H$ be a subhypergroup of $G$.
We say that $H$ is \emph{strongly normal} when $xHx^{-1} \subset H$
holds for all $x \in G$. In this case, we also have $xH = Hx$ for all
$x \in G$, hence the double coset space $H\backslash G/H$
coincides with $G/H$. Also note that $G/H$ is a group by
\cite[Lemma 1.5.2]{MR4696701}.

\begin{prop} \label{prop:stable kernel of free product}
Let $(\map{\phi_{\lbd}}{G_{\lbd}}{H_{\lbd}})_{\lbd \in \Lambda}$
be a family of morphisms of hypergroups. If $\stker \phi_{\lbd}$ is strongly normal for each $\lbd \in \Lambda$,
$\stker \bigast_{\lbd \in \Lambda}\phi_{\lbd}$ is also strongly
normal and we have
$(\bigast_{\lbd \in \Lambda} G_{\lbd})/\stker(\bigast_{\lbd \in \Lambda}\phi_{\lbd}) \cong \bigast_{\lbd} (G_{\lbd}/\stker\phi_{\lbd})$.
\end{prop}

We divide the proof of this proposition into several lemmas.
In the following, $\bigast_{\lbd \in \Lambda} G_{\lbd}, \bigast_{\lbd \in \Lambda} H_{\lbd}$ and $\bigast_{\lbd \in \Lambda} \phi_{\lbd}$
are denoted by $G,H$ and $\phi$ respectively. We regard each
$G_{\lbd}$ as a subhypergroup of $G$ in the canonical way.

\begin{lemm} \label{lemm:candidates}
The stable kernel $\stker \phi$ contains
the following subsets of $G$.
\begin{enumerate}
 \item  $\stker \phi_{\lbd}$ for all $\lbd \in \Lambda$.
 \item $x_1x_2\cdots x_nzy_ny_{n - 1}\cdots y_1$ for
$z \in \stker \phi_{\lbd}$ and $x_k,y_k \in G_{\lbd_k}$ such that
$x_k\stker \phi_{\lbd_k} = \bar{y_k}\stker \phi_{\lbd_k}$.
 \item A product of finitely many subsets of the form in (ii).
\end{enumerate}
\end{lemm}
\begin{proof}
(i) This follows from that $G_{\lbd} \cap \stker \phi$ is a
$\phi_{\lbd}$-stable subhypergroup.

(ii) By induction on $n$. The case of $n = 0$ is (i).
Assume the statement holds for $n$.
Fix $\lbd \in \Lambda$ and consider the following property
on $z \in \stker\phi_{\lbd}$:
for any sequence $(\lambda_1,\lbd_2,\cdots,\lbd_{n + 1})$
and any sequences $(x_1,x_2,\cdots,x_n), (y_1,y_2,\cdots,y_n)$
with $x_k,y_k \in G_{\lbd_k}$ and $x_k\stker \phi_k = \bar{y_k}\stker \phi_k$, the product
$x_1x_2\cdots x_{n + 1}zy_{n + 1}\cdots y_2y_1$ is contained in $\stker\phi$.
We can see that the subset of elements satisfying
this property is closed under the multiplication and the inverse function.
To see that it is also closed under the relation $\sim_{\phi_{\lbd}}$ in Lemma \ref{lemm:equivalence relation}, note that only the case
$\lbd_1\neq\lbd_2\neq\cdots\neq \lbd_{n + 1} \neq \lbd$ is non-trivial
since $G_{\lbd}/\stker \phi_{\lbd}$ is a group for any $\lbd$.
In this case $x_1x_2\cdots x_{n + 1}zy_{n + 1}\cdots y_2y_1$
is an element for any $z$,
hence we can directly see that the set is closed under $\sim_{\phi_{\lbd}}$.

Therefore it forms a $\phi_{\lbd}$-stable
subhypergroup unless it is empty. But the induction hypothesis and
the strong normality of $\stker \phi_{\lbd_{n + 1}}$ implies
$e \in G_{\lbd}$ is contained in this set. Hence we obtain the statement.

(iii) This follows from that $\stker \phi$ is closed under
the multiplication.
\end{proof}

\begin{lemm} \label{lemm:characterization of stability}
Let $\map{\phi}{G}{H}$ be a morphism of hypergroup and $N$ be a
subhypergroup of $G$. Then $N$ is $\phi$-stable
if and only if $\phi$ induces a partition of $\phi(G)/\phi(N)$
labelled by $G/N$. If $N$ is strongly normal additionally,
$\phi(N)$ is also strongly normal in $\phi(G)$ and
$\phi$ induces an isomorphism $G/N\cong \phi(G)/\phi(N)$.
\end{lemm}
\begin{proof}
If $\phi$ induces the partition,
$g \in G\setminus N$ satisfies $\phi(g)\cap \phi(N) = \emptyset$.
Hence $N$ is $\phi$-stable.

Conversely assume that $N$ is $\phi$-stable. Take $g,g' \in G$
so that $\phi(gN) \cap \phi(g'N) \neq \emptyset$. Then
$\phi(g)\cap \phi(g')\phi(n)$ for some $n \in N$, and
$\phi(n)\cap \phi(g'^{-1})\phi(g)\neq \emptyset$. Then the
$\phi$-stability implies $g'^{-1}g \cap N \neq \emptyset$.
Hence $g \in g'N$, which means $gN = g'N$.

To show the latter statement, take $h \in \phi(G)$
and $g \in G$ such that $h \in \phi(g)$.
Then we have $h\phi(N)\bar{h} \subset \phi(g)\phi(N)\phi(\bar{g}) = \phi(gN\bar{g}) = \phi(N)$. Hence $\phi(N) \subset \phi(G)$ is strongly normal.
To prove $G/N\cong \phi(G)/\phi(N)$, take $g \in G$ and $h\phi(N), h'\phi(N) \in \phi(gN)$. Then $h\phi(N)\bar{h'\phi(N)} \subset \phi(gN)\bar{\phi(gN)} = \phi(N)$, which means $h\phi(N) = h'\phi(N)$ since $\phi(G)/\phi(N)$ is a group. Hence $\phi$ induces a single-valued surjective
map from $G/N$ to $\phi(G)/\phi(N)$, which is injective since
it gives the partition.
\end{proof}

\begin{proof}[Proof of Proposition \ref{prop:stable kernel of free product}]
 Let $N$ be the union of all subsets of the form (iii) in Lemma \ref{lemm:candidates}. This is a strongly normal hypersubgroup of $G$, containing
$\stker \phi_{\lbd}$ for all $\lbd \in \Lambda$.
Hence the universal property means that there is a canonical morphism
from $\bigast_{\lbd \in \Lambda} (G_{\lbd}/\stker \phi_{\lbd})$ to
$G/N$. On the other hand, we have the canonical mophism from
$G$ to $\bigast_{\lbd \in \Lambda} (G_{\lbd}/\stker \phi_{\lbd})$. Then the strong normality of each $\stker \phi_{\lbd}$ implies that
this factors through $G\longrightarrow G/N$.
Hence we have $\bigast_{\lbd} G_{\lbd}/\stker \phi_{\lbd} \cong G/N$.

The remaining part is
the $\phi$-stability of $N$, which concludes the statement since
$N$ is contained in $\stker \phi$ by constructon.
Lemma \ref{lemm:characterization of stability} implies we have
$G/N \cong \bigast_{\lbd \in \Lambda} G_{\lbd}/\stker\phi_{\lbd}\cong \bigast_{\lbd\in \Lambda} \phi(G_{\lbd})/\phi(\stker\phi_{\lbd})$.
Hence $\phi(G)\cong \bigast_{\lbd \in \Lambda} \phi_{\lbd}(G_{\lbd})\longrightarrow \bigast_{\lbd\in \Lambda} \phi(G_{\lbd})/\phi(\stker\phi_{\lbd})$ factors throught $\phi(G)\longrightarrow \phi(G)/\phi(N)$.
On the other hand we also have the inverse map by the universal property.
Hence we have $G/N\cong \phi(G)/\phi(N)$, which shows the
$\phi$-stability of $N$ by Lemma \ref{lemm:characterization of stability}.
\end{proof}

In general the stable kernel of the unimodularization of
a discrete quantum group is not strongly normal. For example
this is not the case for the dual of non-abelian
compact group, in which the stable kernel is trivial.

Before going to an application, we give an description
$\Gamma_{\hyp}/\stker q_{\hyp}$ in light of $\rep^{\fin}\hat{\Gamma}$.
See \cite[Subsection 1.1]{MR3782061} for the definition of the chain group
$\mathrm{Ch}[\cl{C}]$ of a rigid \cstar-tensor category $\cl{C}$.

\begin{lemm} \label{lemm:comparison with the chain group}
Let $\Gamma$ be a discrete quantum group such that
the stable kernel $\stker q_{\hyp}$
of the unimodularization is strongly normal.
Then $\Gamma_{\hyp}/\stker q_{\hyp}$ is isomorphic to
$\mathrm{Ch}[\rep^{\fin}\hat{\Gamma}]$.
Moreover there is a morphism $\Gamma \longrightarrow \Gamma_{\hyp}/\stker q_{\hyp}$ which induces the original morphism $\Gamma_{\hyp}\longrightarrow \Gamma_{\hyp}/\stker q_{\hyp}$.
\end{lemm}
\begin{proof}
It is not difficult to see that $\stker q_{\hyp}$ gives
a natural grading on $\rep^{\fin}\hat{\Gamma}$.

Let $G$ be a group and consider a $G$-grading on $\rep^{\fin}\hat{\Gamma}$. We have a Hopf \astar-algebra homomorphism $\cl{O}(\hat{\Gamma}) \longrightarrow \C[G]$ defined as follows:
\begin{align*}
 (\xi^*\tensor 1)U_{\pi}(\eta\tensor 1) \longmapsto \ip{\xi,\eta}g,
\end{align*}
where $\xi,\eta \in H_{\pi}$ and $g$ is
the grading of $\pi \in \Gamma_{\hyp}$. Hence we have a morphism
$\Gamma\longrightarrow G$, which induces a morphism
$\map{\phi}{\Gamma_{\hyp}}{G}$ of hypergroup defined as
$\phi(\pi) = g$ when the grading of $\pi$ is $g$.
Since $G$ is unimodular, this morphism factors through
$\map{q_{\hyp}}{\Gamma_{\hyp}}{(\Gamma_{\uni})_{\hyp}}$.
In particular the stable kernel of $\phi$, which coincides
with $\{\pi \in \Gamma_{\hyp}\mid \phi(\pi) = \{e\}\}$,
contains $\stker q_{\hyp}$. Now \cite[Lemma 1.1]{MR3782061} implies
$\mathrm{Ch}{\rep^{\fin}\hat{\Gamma}}\cong \Gamma_{\hyp}/\stker q_{\hyp}$. The latter statement follows from this discussion.
\end{proof}

\begin{coro} \label{coro:findexsqg}
Let $(K_{\lbd})_{\lbd \in \Lambda}$ be a family
of connected simply connected compact Lie groups.
The associated weight lattice and root lattice is denoted by $P_{\lbd}$
and $Q_{\lbd}$ respectively for each $\lbd \in \Lambda$.
Let $\map{\phi}{\bigast_{\lbd \in \Lambda}{\hat{K_{\lbd}}}}{\bigast_{\lbd \in \Lambda}P_{\lbd}/Q_{\lbd}}$ be the morphism induced from
the $P_{\lbd}/Q_{\lbd}$-grading on $\rep^{\fin}K_{\lbd}$.
Then $H \longmapsto \phi^{-1}(H)$ gives a bijection between the set of finite index quantum subgroups of $\bigast_{\lbd \in \Lambda} \hat{K_{\lbd}}$ and the set of finite index subgroups
of $\bigast_{\lbd \in \Lambda} P_{\lbd}/Q_{\lbd}$. Moreover
it preserves the indices of quantum subgroups.
\end{coro}
\begin{proof}
By the same reason with Example \ref{exam:dual of cLg},
we can replace $K_{\lbd}$ by the Drinfeld-Jimbo deformation
$K_{\lbd,q}$ with $0 < q < 1$.
Let $\map{q_{\lbd}}{\hat{K_{\lbd,q}}}{\hat{K_{\lbd,q}}_{\uni}}$
be the unimodularization.
Then, as discussed in Example \ref{exam:dual of cLg},
$\stker q_{\lbd}\backslash \hat{K_{\lbd,q}}/\stker q_{\lbd}\cong P_{\lbd}/Q_{\lbd}$. This implies that $\stker q_{\lbd}$ is strongly normal.
Hence we can apply Proposition \ref{prop:stable kernel of free product} and Lemma \ref{lemm:comparison with the chain group}
to our case to obtain the conclusion, combining with Theorem \ref{thrm:reduction to the unimodularization}. The last statement
follows from the comparison of the dual left coideals.
\end{proof}

In the same way we also have:

\begin{coro}[{c.f. \cite[Main Theorem]{MR4824928}}]
Let $\{U_{Q_{\lbd}}^+\}_{\lbd \in \Lambda}$ be a family of
free unitary quantum group. Then there is a natural
bijection between the set
of finite index quantum subgroups of $\bigast_{\lbd \in \Lambda} \hat{U_{Q_{\lbd}}^+}$ and the set of finite index subgroups of $\Z^{\ast \Lambda}$, which preserves the indices of quantum subgroups.
\end{coro}

At the end of this paper, we give a concrete description of
subhypergroups corresponding to finite index quantum subgroups of
$\hat{SU(2)}\ast\hat{SU(2)}$.

\begin{exam}
 Consider $\Gamma = \hat{SU(2)}\ast\hat{SU(2)}$.
In this case we can classify finite index quantum subgroups of $\Gamma$
by finite index subgroups of $G = \Z/2\Z\ast \Z/2\Z$.
Let $s,t$ be the canonical generators of the free product. Then
any element of $G$ can be uniquely presented in the either form of
$stst\cdots$ or $tsts\cdots$. Hence the following
is the complete list of finite index subgroups of $G$:
\begin{itemize}
 \item $\Z/2\Z\ast \Z/2\Z$.
 \item $H_k := \{(st)^{kn}\mid n \in \Z\}$ for a positive integer $k$.
 \item $H_k^s := H_k\cup tH_k$ for an integer $k \ge 2$.
 \item $H_k^t := H_k\cup sH_k$ for an integer $k \ge 2$.
 \item $H_{k}^{s,\mathrm{odd}} := H_{2k}\cup t(st)^kH_{2k}$ for a positive integer $k$.
 \item $H_{k}^{t,\mathrm{odd}} := H_{2k}\cup s(ts)^kH_{2k}$ for a positive integer $k$.
\end{itemize}
By Corollary \ref{coro:findexsqg}, this list also classifies
finite index quantum subgroups of $\hat{SU(2)}\ast\hat{SU(2)}$.
\end{exam}
\vspace{10pt}
\noindent
{\bf Acknowlegements.}
The author appreciates to Yasuyuki Kawahigashi for helpful comments
and continuous support. He is grateful to Kan Kitamura for discussion
on finite dimensional right $\cl{O}(\G)$-module \cstar-algebras.
He is also grateful to Keisuke Hoshino for discussion on categories
of hypergroups.



\begin{bibdiv}
\begin{biblist}

\bib{MR1484551}{article}{
      author={Banica, Teodor},
       title={Le groupe quantique compact libre {${\rm U}(n)$}},
        date={1997},
        ISSN={0010-3616,1432-0916},
     journal={Comm. Math. Phys.},
      volume={190},
      number={1},
       pages={143\ndash 172},
         url={https://doi.org/10.1007/s002200050237},
      review={\MR{1484551}},
}

\bib{MR945550}{article}{
      author={Baillet, Michel},
      author={Denizeau, Yves},
      author={Havet, Jean-Fran\c~cois},
       title={Indice d'une esp\'erance conditionnelle},
        date={1988},
        ISSN={0010-437X,1570-5846},
     journal={Compositio Math.},
      volume={66},
      number={2},
       pages={199\ndash 236},
         url={http://www.numdam.org/item?id=CM_1988__66_2_199_0},
      review={\MR{945550}},
}

\bib{MR3308880}{book}{
      author={Bischoff, Marcel},
      author={Kawahigashi, Yasuyuki},
      author={Longo, Roberto},
      author={Rehren, Karl-Henning},
       title={Tensor categories and endomorphisms of von {N}eumann
  algebras---with applications to quantum field theory},
      series={SpringerBriefs in Mathematical Physics},
   publisher={Springer, Cham},
        date={2015},
      volume={3},
        ISBN={978-3-319-14300-2; 978-3-319-14301-9},
         url={https://doi.org/10.1007/978-3-319-14301-9},
      review={\MR{3308880}},
}

\bib{MR3782061}{article}{
      author={Bichon, Julien},
      author={Neshveyev, Sergey},
      author={Yamashita, Makoto},
       title={Graded twisting of comodule algebras and module categories},
        date={2018},
        ISSN={1661-6952,1661-6960},
     journal={J. Noncommut. Geom.},
      volume={12},
      number={1},
       pages={331\ndash 368},
         url={https://doi.org/10.4171/JNCG/278},
      review={\MR{3782061}},
}

\bib{MR3863479}{article}{
      author={Chirvasitu, Alexandru},
       title={Relative {F}ourier transforms and expectations on coideal
  subalgebras},
        date={2018},
        ISSN={0021-8693,1090-266X},
     journal={J. Algebra},
      volume={516},
       pages={271\ndash 297},
         url={https://doi.org/10.1016/j.jalgebra.2018.08.033},
      review={\MR{3863479}},
}

\bib{MR4627100}{article}{
      author={Chirvasitu, Alexandru},
      author={Krajczok, Jacek},
      author={So\l~tan, Piotr~M.},
       title={Compact quantum group structures on type-{I} {$\rm
  C^*$}-algebras},
        date={2023},
        ISSN={1661-6952,1661-6960},
     journal={J. Noncommut. Geom.},
      volume={17},
      number={3},
       pages={1129\ndash 1143},
         url={https://doi.org/10.4171/jncg/516},
      review={\MR{4627100}},
}

\bib{MR3285870}{article}{
      author={Daws, Matthew},
       title={Remarks on the quantum {B}ohr compactification},
        date={2013},
        ISSN={0019-2082,1945-6581},
     journal={Illinois J. Math.},
      volume={57},
      number={4},
       pages={1131\ndash 1171},
         url={http://projecteuclid.org/euclid.ijm/1417442565},
      review={\MR{3285870}},
}

\bib{MR3675047}{incollection}{
      author={De~Commer, Kenny},
       title={Actions of compact quantum groups},
        date={2017},
   booktitle={Topological quantum groups},
      series={Banach Center Publ.},
      volume={111},
   publisher={Polish Acad. Sci. Inst. Math., Warsaw},
       pages={33\ndash 100},
      review={\MR{3675047}},
}

\bib{MR4776189}{article}{
      author={De~Commer, Kenny},
      author={Dzokou~Talla, Joel~Right},
       title={Invariant integrals on coideals and their {D}rinfeld doubles},
        date={2024},
        ISSN={1073-7928,1687-0247},
     journal={Int. Math. Res. Not. IMRN},
      number={14},
       pages={10650\ndash 10677},
         url={https://doi.org/10.1093/imrn/rnae094},
      review={\MR{4776189}},
}

\bib{MR3121622}{article}{
      author={De~Commer, Kenny},
      author={Yamashita, Makoto},
       title={Tannaka-{K}re\u{\i}n duality for compact quantum homogeneous
  spaces. {I}. {G}eneral theory},
        date={2013},
        ISSN={1201-561X},
     journal={Theory Appl. Categ.},
      volume={28},
       pages={No. 31, 1099\ndash 1138},
      review={\MR{3121622}},
}

\bib{MR2980506}{article}{
      author={Daws, Matthew},
      author={Kasprzak, Pawe\l},
      author={Skalski, Adam},
      author={So\l~tan, Piotr~M.},
       title={Closed quantum subgroups of locally compact quantum groups},
        date={2012},
        ISSN={0001-8708,1090-2082},
     journal={Adv. Math.},
      volume={231},
      number={6},
       pages={3473\ndash 3501},
         url={https://doi.org/10.1016/j.aim.2012.09.002},
      review={\MR{2980506}},
}

\bib{MR2183279}{article}{
      author={Etingof, Pavel},
      author={Nikshych, Dmitri},
      author={Ostrik, Viktor},
       title={On fusion categories},
        date={2005},
        ISSN={0003-486X,1939-8980},
     journal={Ann. of Math. (2)},
      volume={162},
      number={2},
       pages={581\ndash 642},
         url={https://doi.org/10.4007/annals.2005.162.581},
      review={\MR{2183279}},
}

\bib{freslon2024discretequantumsubgroupsfree}{misc}{
      author={Freslon, Amaury},
      author={Weber, Moritz},
       title={Discrete quantum subgroups of free unitary quantum groups},
        date={2024},
         url={https://arxiv.org/abs/2403.01526},
}

\bib{MR4824928}{article}{
      author={Hoshino, Mao},
      author={Kitamura, Kan},
       title={A note on quantum subgroups of free quantum groups},
        date={2024},
        ISSN={1631-073X,1778-3569},
     journal={C. R. Math. Acad. Sci. Paris},
      volume={362},
       pages={1327\ndash 1330},
      review={\MR{4824928}},
}

\bib{MR4728596}{article}{
      author={Hoshino, Mao},
       title={The automatic imprimitivity for {$G_q$}},
        date={2024},
        ISSN={0022-1236,1096-0783},
     journal={J. Funct. Anal.},
      volume={286},
      number={12},
       pages={Paper No. 110413, 20},
         url={https://doi.org/10.1016/j.jfa.2024.110413},
      review={\MR{4728596}},
}

\bib{MR4742819}{article}{
      author={Kitamura, Kan},
       title={Discrete quantum subgroups of complex semisimple quantum groups},
        date={2024},
        ISSN={1073-7928,1687-0247},
     journal={Int. Math. Res. Not. IMRN},
      number={9},
       pages={7234\ndash 7254},
         url={https://doi.org/10.1093/imrn/rnad117},
      review={\MR{4742819}},
}

\bib{MR3128415}{article}{
      author={Kasprzak, Pawe\l},
      author={So\l~tan, Piotr~M.},
       title={Embeddable quantum homogeneous spaces},
        date={2014},
        ISSN={0022-247X,1096-0813},
     journal={J. Math. Anal. Appl.},
      volume={411},
      number={2},
       pages={574\ndash 591},
         url={https://doi.org/10.1016/j.jmaa.2013.07.084},
      review={\MR{3128415}},
}

\bib{MR1832993}{article}{
      author={Kustermans, Johan},
      author={Vaes, Stefaan},
       title={Locally compact quantum groups},
        date={2000},
        ISSN={0012-9593},
     journal={Ann. Sci. \'Ecole Norm. Sup. (4)},
      volume={33},
      number={6},
       pages={837\ndash 934},
         url={https://doi.org/10.1016/S0012-9593(00)01055-7},
      review={\MR{1832993}},
}

\bib{nakamura2024categorieshypermagmashypergroupsrelated}{misc}{
      author={Nakamura, So},
      author={Reyes, Manuel~L.},
       title={Categories of hypermagmas, hypergroups, and related
  hyperstructures},
        date={2024},
         url={https://arxiv.org/abs/2304.09273},
}

\bib{MR3204665}{book}{
      author={Neshveyev, Sergey},
      author={Tuset, Lars},
       title={Compact quantum groups and their representation categories},
      series={Cours Sp\'ecialis\'es [Specialized Courses]},
   publisher={Soci\'et\'e{} Math\'ematique de France, Paris},
        date={2013},
      volume={20},
        ISBN={978-2-85629-777-3},
      review={\MR{3204665}},
}

\bib{MR3556413}{article}{
      author={Neshveyev, Sergey},
      author={Yamashita, Makoto},
       title={Classification of non-{K}ac compact quantum groups of {${\rm
  SU}(n)$} type},
        date={2016},
        ISSN={1073-7928,1687-0247},
     journal={Int. Math. Res. Not. IMRN},
      number={11},
       pages={3356\ndash 3391},
         url={https://doi.org/10.1093/imrn/rnv241},
      review={\MR{3556413}},
}

\bib{MR3933035}{article}{
      author={Neshveyev, Sergey},
      author={Yamashita, Makoto},
       title={Categorically {M}orita equivalent compact quantum groups},
        date={2018},
        ISSN={1431-0635,1431-0643},
     journal={Doc. Math.},
      volume={23},
       pages={2165\ndash 2216},
      review={\MR{3933035}},
}

\bib{MR1048073}{article}{
      author={Parshall, Brian},
      author={Wang, Jian~Pan},
       title={Quantum linear groups},
        date={1991},
        ISSN={0065-9266,1947-6221},
     journal={Mem. Amer. Math. Soc.},
      volume={89},
      number={439},
       pages={vi+157},
         url={https://doi.org/10.1090/memo/0439},
      review={\MR{1048073}},
}

\bib{MR2210362}{article}{
      author={So\l~tan, Piotr~M.},
       title={Quantum {B}ohr compactification},
        date={2005},
        ISSN={0019-2082,1945-6581},
     journal={Illinois J. Math.},
      volume={49},
      number={4},
       pages={1245\ndash 1270},
         url={http://projecteuclid.org/euclid.ijm/1258138137},
      review={\MR{2210362}},
}

\bib{MR549940}{article}{
      author={Takeuchi, Mitsuhiro},
       title={Relative {H}opf modules---equivalences and freeness criteria},
        date={1979},
        ISSN={0021-8693},
     journal={J. Algebra},
      volume={60},
      number={2},
       pages={452\ndash 471},
         url={https://doi.org/10.1016/0021-8693(79)90093-0},
      review={\MR{549940}},
}

\bib{MR2276175}{article}{
      author={Tomatsu, Reiji},
       title={Amenable discrete quantum groups},
        date={2006},
        ISSN={0025-5645,1881-1167},
     journal={J. Math. Soc. Japan},
      volume={58},
      number={4},
       pages={949\ndash 964},
         url={http://projecteuclid.org/euclid.jmsj/1179759531},
      review={\MR{2276175}},
}

\bib{MR2335776}{article}{
      author={Tomatsu, Reiji},
       title={A characterization of right coideals of quotient type and its
  application to classification of {P}oisson boundaries},
        date={2007},
        ISSN={0010-3616,1432-0916},
     journal={Comm. Math. Phys.},
      volume={275},
      number={1},
       pages={271\ndash 296},
         url={https://doi.org/10.1007/s00220-007-0267-6},
      review={\MR{2335776}},
}

\bib{MR1814995}{article}{
      author={Vaes, Stefaan},
       title={The unitary implementation of a locally compact quantum group
  action},
        date={2001},
        ISSN={0022-1236,1096-0783},
     journal={J. Funct. Anal.},
      volume={180},
      number={2},
       pages={426\ndash 480},
         url={https://doi.org/10.1006/jfan.2000.3704},
      review={\MR{1814995}},
}

\bib{MR2182592}{article}{
      author={Vaes, Stefaan},
       title={A new approach to induction and imprimitivity results},
        date={2005},
        ISSN={0022-1236,1096-0783},
     journal={J. Funct. Anal.},
      volume={229},
      number={2},
       pages={317\ndash 374},
         url={https://doi.org/10.1016/j.jfa.2004.11.016},
      review={\MR{2182592}},
}

\bib{MR1382726}{article}{
      author={Van~Daele, Alfons},
      author={Wang, Shuzhou},
       title={Universal quantum groups},
        date={1996},
        ISSN={0129-167X,1793-6519},
     journal={Internat. J. Math.},
      volume={7},
      number={2},
       pages={255\ndash 263},
         url={https://doi.org/10.1142/S0129167X96000153},
      review={\MR{1382726}},
}

\bib{MR3119236}{article}{
      author={Wang, Shuzhou},
       title={Equivalent notions of normal quantum subgroups, compact quantum
  groups with properties {$F$} and {$FD$}, and other applications},
        date={2014},
        ISSN={0021-8693,1090-266X},
     journal={J. Algebra},
      volume={397},
       pages={515\ndash 534},
         url={https://doi.org/10.1016/j.jalgebra.2013.09.014},
      review={\MR{3119236}},
}

\bib{MR1316765}{article}{
      author={Wang, Shuzhou},
       title={Free products of compact quantum groups},
        date={1995},
        ISSN={0010-3616,1432-0916},
     journal={Comm. Math. Phys.},
      volume={167},
      number={3},
       pages={671\ndash 692},
         url={http://projecteuclid.org/euclid.cmp/1104272163},
      review={\MR{1316765}},
}

\bib{MR996807}{article}{
      author={Watatani, Yasuo},
       title={Index for {$C^*$}-subalgebras},
        date={1990},
        ISSN={0065-9266,1947-6221},
     journal={Mem. Amer. Math. Soc.},
      volume={83},
      number={424},
       pages={vi+117},
         url={https://doi.org/10.1090/memo/0424},
      review={\MR{996807}},
}

\bib{MR901157}{article}{
      author={Woronowicz, S.~L.},
       title={Compact matrix pseudogroups},
        date={1987},
        ISSN={0010-3616,1432-0916},
     journal={Comm. Math. Phys.},
      volume={111},
      number={4},
       pages={613\ndash 665},
         url={http://projecteuclid.org/euclid.cmp/1104159726},
      review={\MR{901157}},
}

\bib{MR1616348}{incollection}{
      author={Woronowicz, S.~L.},
       title={Compact quantum groups},
        date={1998},
   booktitle={Sym\'etries quantiques ({L}es {H}ouches, 1995)},
   publisher={North-Holland, Amsterdam},
       pages={845\ndash 884},
      review={\MR{1616348}},
}

\bib{MR4696701}{book}{
      author={Zieschang, Paul-Hermann},
       title={Hypergroups},
   publisher={Springer, Cham},
        date={2023},
        ISBN={978-3-031-39488-1; 978-3-031-39489-8},
         url={https://doi.org/10.1007/978-3-031-39489-8},
      review={\MR{4696701}},
}

\end{biblist}
\end{bibdiv}

\end{document}